\documentclass[11pt]{article}
\usepackage[margin=1in]{geometry}

\usepackage{amssymb}
\usepackage{amsthm}
\usepackage{amscd}
\usepackage{amsfonts}
\usepackage{graphicx}
\usepackage{fancyhdr}
\usepackage{epsfig}
\usepackage{hyperref}
\usepackage{makeidx}
\usepackage{setspace}
\usepackage{subfigure}
\usepackage{empheq}

\usepackage{verbatim}
\usepackage{xcolor}
\usepackage{lineno}
\usepackage{tikz-cd}
\usepackage[lastexercise]{exercise}
\usepackage[]{tcolorbox}
\usepackage{bm}
\tcbuselibrary{breakable}

\usepackage{tikz}
\usetikzlibrary{positioning}
\usepackage{pgfplots}
\usetikzlibrary{matrix}
\theoremstyle{plain} \numberwithin{equation}{section}
\newtheorem{theorem}{Theorem}[section]
\newtheorem{corollary}[theorem]{Corollary}

\newtheorem{lemma}[theorem]{Lemma}
\newtheorem{proposition}[theorem]{Proposition}

\theoremstyle{definition}
\newtheorem{definition}[theorem]{Definition}

\newtheorem{remark}[theorem]{Remark}

\newtheorem{notation}[theorem]{Notation}

\newcommand{\C}{\mathcal{C}}

\newcommand{\F}{\mathcal{F}}
\newcommand{\mcL}{\mathcal{L}}
\newcommand{\mcQ}{\mathcal{Q}}
\newcommand{\mcM}{\mathcal M}
\newcommand{\mcN}{\mathcal N}

\newcommand{\NN}{\mathbb N}
\newcommand{\PF}{\mathbb{PF}}
\newcommand{\PE}{\mathbb{P}E}

\newcommand{\PoE}{\mathbb{P}_oE}
\newcommand{\PP}{\mathbb P}

\newcommand{\FF}{\mathbb F}
\newcommand{\mcS}{\mathcal S}

\newcommand{\cl}{\mathrm cl}

\newcommand{\rk}{\mathrm{rk}}

\newcommand{\rowsp}{\mathrm{rowsp}}
\newcommand{\colsp}{\mathrm{colsp}}

\newcounter{alp}
\newcounter{ara}
\newcounter{rom}


\begin{document}
\title{The Projectivization Matroid of a $q$-Matroid}
\author{Benjamin Jany \footnote{Department of Mathematics, University of Kentucky, Lexington KY 40506-0027, USA; benjamin.jany@uky.edu.} }
\date{May 4, 2022}

\maketitle

\begin{abstract}
In this paper, we investigate the relation between a $q$-matroid and its associated matroid called the projectivization matroid.
The latter arises by projectivizing the groundspace of the $q$-matroid and considering the projective space as the groundset of the associated matroid on which is defined a rank function compatible with that of the $q$-matroid. We show that the projectivization map is a functor from categories of $q$-matroids to categories of matroids, which allows to prove new results about maps of $q$-matroids. We furthermore show the characteristic polynomial of a $q$-matroid is equal to that of the projectivization matroid. We use this relation to establish a recursive formula for the characteristic polynomial of a $q$-matroid in terms of the characteristic polynomial of its minors. Finally we use the projectivization matroid to prove a $q$-analogue of the critical theorem in terms of $\FF_{q^m}$-linear rank metric codes and $q$-matroids.
\end{abstract}


\section{Introduction}

In recent years, $q$-matroids, the $q$-analogue of matroids, have been intensively studied due to their connection to linear rank metric codes. They were first studied by Jurrius and Pellikaan in \cite{JP18}, who showed that an $\FF_{q^m}$-linear rank metric code induces a $q$-matroid. It was shown later on that matrix linear rank metric codes induce a $q$-polymatroid, a generalization of $q$-matroids (see \cite{ GLJ21qpolymatroids,gorla2020rank}).
Since then, many results on $q$-(poly)matroids, and how they relate to rank metric codes, have been established, see for example \cite{cyclicflats,BCIJ21,BCJ21,GLJ21,GLJ21qpolymatroids,gorla2020rank,JPV21}.
Because of their $q$-analogue nature, many of the newly discovered properties of $q$-matroids turn out to be analogues of well established matroid theory results. See \cite{crapo1970foundations, Greene,oxley,RotaMatroid} for more information on matroid theory.
It has therefore been of interest to determine which notions and properties of matroid theory can be generalized to $q$-matroids.

In \cite{BCJ21}, the authors show that, similarly to matroids, there exists a variety of cryptomorphic definitions for $q$-matroids. In this paper, we define $q$-matroids via a rank function on the lattice of subspace of a finite dimensional vector space over a finite field, and occasionally use the flat cryptomorphism. 
In \cite{BCIJ21}, Byrne and co-authors, define the notion of a characteristic polynomial for $q$-polymatroids and use it to establish a $q$-analogue of the Assmus-Mattson Theorem. Furthermore they show that the characteristic polynomial of a $q$-polymatroid induced by a linear rank metric code determines the weight distribution of the code.
Maps between $q$-matroids are defined and studied in \cite{GLJ21}, which allows the authors to consider $q$-matroids from a category theory perspective. They introduce the notions of weak and strong maps, which respectively respect the rank structure and the flat structure of $q$-matroids. Although those maps are defined in an analogous way than the weak and strong maps between matroids (see \cite{Heunen2018TheCO, RotaMatroid}), substantial differences occur when comparing categories of $q$-matroids with categories of matroids. In fact, the authors show that, unlike for categories of matroids, a coproduct does not always exist in the category of $q$-matroids with strong maps but always exists when the morphisms are linear weak maps.

 In \cite{JPV21}, Johnsen and co-authors make the connection between matroids and $q$-matroids more apparent by showing that a $q$-matroid induces a matroid, called the projectivization matroid. Furthermore, they show that the lattice of flats of the $q$-matroid is isomorphic to the lattice of flats of its projectivization matroid. This allows them to express the generalized rank weights of an $\FF_{q^m}$-linear rank metric code in terms of the Betti numbers of the dual of the projectivization matroid.  

In this paper, we further investigate the construction introduced in \cite{JPV21} and use the projectivization matroid as a tool to study
maps between $q$-matroids and the characteristic polynomial of a $q$-matroid. We define and study properties of the projectivization matroid in Section 3. 
In Section 4, we show that the projectivization map from a vector space to its projective space is a functor from the category of $q$-matroids with weak (resp.\ strong) maps to the category of matroids with weak (resp.\ strong) maps. We use the relation between those categories to show that strong maps between $q$-matroids are weak maps.
Although Section 4 shows how the projectivization matroid can be used in a more category theory approach, results therein are not used in later sections. The reader interested in the relation between the characteristic polynomial of a $q$-matroid and that of its projectivization matroid may skip Section 4 on a first reading.
We then proceed in Section 5 to study the characteristic polynomial of $q$-matroids. We start by showing that the characteristic polynomial is identically $0$ if the $q$-matroid contains a loop, and is fully determined by the lattice of flats otherwise. We use this fact to show that the characteristic polynomial of a $q$-matroid is equal to that of its projectivization matroid. This in turn, allows us to find a recursive formula for the characteristic polynomial of a $q$-matroid in terms of the characteristic polynomial of its minors. 
Finally, in Section 6, we consider the projectivization matroid of $q$-matroids induced by an $\FF_{q^m}$-linear rank metric code. In \cite{ABNR21}, Alfarano et.\ al.\ associate a linear block code with the Hamming metric to an $\FF_{q^m}$-linear rank metric code $\C$. This code, called a Hamming-metric code associated to $\C$, induces a matroid that turns out to be equivalent to the projectivization matroid of the $q$-matroid associated with $\C$. This connection allows us to prove a $q$-analogue of the critical theorem in terms of $\FF_{q^m}$-linear rank metric codes and $q$-matroids.

\vspace{.2cm}

\underline{\textbf{Notation:}}
Throughout  $\FF_q$ denotes a finite field of order $q$. $E$ denotes a finite dimensional vector space over $\FF_q$, and $\mcL(E)$ denotes the lattice of subspace of $E$. $S$ and $T$  are finite sets, $2^{S}$ is the power set of $S$ and $[n] := \{1, \ldots, n\}$ for $n \in \NN_0$. Furthermore given a set $S$ and $A \subseteq S$, let $S-A := \{e \in S \, : \, e \notin A\}$. 
Finally, $q$-matroids will be denoted by the script letters $\mcM$, $\mcN$, whereas matroids will be denoted by the capital letters $M, N$.

\section{Basic Notions of Matroids and $q$-Matroids}
In this section we review well-known notions of matroids and $q$-matroids that will be used throughout the paper. For more details about matroids and $q$-matroids the reader may refer to \cite{BCJ21, JP18, oxley, RotaMatroid}.

\begin{definition}
A \emph{matroid} is an ordered pair $M = (S, r)$, where $S$ is a finite set and $r$ is a function $r : 2^S \rightarrow \NN_0$ such that for all $A, B \in 2^{S}$ :
\begin{enumerate}
    \item[(R1)] Boundedness: $0 \leq r(A) \leq |A|$.
    \item[(R2)] Monotonicity: If $A \subseteq B$ then $r(A) \leq r(B)$.
    \item[(R3)] Submodularity: $r(A \cup B) + r(A \cap B) \leq r(A) + r(B)$.
\end{enumerate}
$S$ is called the \emph{groundset} of $M$ and $r$ its \emph{rank function}.
\end{definition}

Throughout, identify $\{e\}$ with $e$ and $\{v\}$ with $v$. Two matroids $M = (S, r_M)$ and $N = (T, r_N)$ are \emph{equivalent}, denoted $M \cong N$, if there exists a bijection between the groundsets, $\psi : S \rightarrow T$, such that $r_M(A) = r_N(\psi(A))$ for all $A \subseteq S$. 
Given a matroid $M = (S,r)$, $e \in S$ is a \emph{loop} of $M$ if $r(e) = 0$. $M$ is said to be \emph{loopless} if it does not contain any loops.
A subset $F \subseteq S$ is a \emph{flat} if $r(F \cup v) > r(F)$ for all $v \notin  F$.  It is well known that the collection of flats, denoted $\F_M$, forms a geometric lattice. 
For any $F_1, F_2 \in \F_M$, the meet and join are defined as follow $F_1 \wedge F_2 := F_1 \cap F_2$ and $F_1 \vee F_2 := \cl_M(F_1 \cup F_2)$, where $\cl_M(A) = \{v \in S \, :\, r(A \cup v) = r(A)\} = \bigcap \{F \in \F_{M} \, : \, A \subseteq F\}$.
Given $F_1, F_2 \in \F_{M}$, we say $F_2$ \emph{covers} $F_1$ if for all $F \in \F_M$ such that $F_1 \subseteq F \subseteq F_2$ then $F = F_1$ or $F = F_2$.
When discussing $\F_M$, we interchangeably use the terms collection of flats and lattice of flats. 
The flats of a matroid satisfy three axiomatic properties that fully determine the matroid.

\begin{proposition}\label{flats matroid}\cite[Sec. 1.4 Prob 11.]{oxley}
Let $M = (S, r_M)$ be a matroid and $\F_M$ its collection of flats. Then $\F_M$ satisfies the following:
\begin{enumerate}
    \item[(F1)] $S \in \F_M$.
    \item[(F2)] If $F_1, F_2 \in \F_M$ then $F_1 \cap F_2 \in \F_M$.
   \item[(F3)] Let $F \in \F_M$ and $v \notin F$, then there exists a unique $F' \in \F_M$ covering $F$ such that $F \cup v \subseteq F'$.
\end{enumerate}
Furthermore, $r_M$ is uniquely determined by $\F_M$ and  $r_M(A) = h(\cl_M(A))$ for $A \subseteq S$, where $h(F)$ denotes the height of $F$ in the lattice $\F_M$.
\end{proposition}

We now turn to $q$-matroids, which are defined in an analogous way. Recall that $E$ denotes a finite dimensional vector space over $\FF_q$ and $\mcL(E)$ is the collection of subspace of $E$.

\begin{definition}
A \emph{$q$-matroid} is an ordered pair $\mcM = (E, \rho)$, where $\rho$ is a function $\rho : \mcL(E) \rightarrow \NN_0$ such that for all $V, W \in \mcL(E)$ :
\begin{enumerate}
    \item[(qR1)] Boundedness: $0 \leq \rho(V) \leq \dim(V)$.
    \item[(qR2)] Monotonicity: If $V \subseteq W$ then $\rho(V) \leq \rho(W)$.
    \item[(qR3)] Submodularity: $\rho(V + W) + \rho(V \cap W) \leq \rho(V) + \rho(W)$.
\end{enumerate}
$E$ is called the \emph{groundspace} of $\mcM$ and $\rho$ its \emph{rank function}.
\end{definition}

 Two $q$-matroids $\mcM = (E_1, \rho_{\mcM})$ and $\mcN = (E_2, \rho_{\mcN})$ are \emph{equivalent}, denoted $\mcM \cong \mcN$, if there exists a linear isomorphism $\psi: E_1 \rightarrow E_2$ such that $\rho_{\mcM}(V) = \rho_{\mcN}(\psi(V))$ for all $V \leq E_1$.
Given a $q$-matroid $\mcM = (E, \rho)$ we say $\langle e \rangle$, where $e \in E - \{0\}$, is a \emph{loop} if $\rho(\langle e \rangle) = 0$, and $\mcM$ is \emph{loopless} if it does not contain any loops. 
A subspace $F \leq E$ is a \emph{flat} of $\mcM$ if $\rho(F + \langle v \rangle) > \rho(F)$ for all $v \notin F$.
Furthermore the collection of flats of a $q$-matroid, denoted $\F_{\mcM}$, forms a geometric lattice as well. The meet and join operation are given by  $F_1 \wedge F_2 := F_1 \cap F_2$ and $F_1 \vee F_2 := \cl_{\mcM}(F_1 + F_2)$, where $\cl_{\mcM}(V) := \{v \, : \, \rho(V + \langle v \rangle) = \rho(V)\} = \bigcap \{F \in \F_{\mcM} \, : \, V \leq F\}$.
The notion of cover for the lattice of flats of $q$-matroids is identical to that of matroids.
The collection of flats of a $q$-matroid also satisfies three axiomatic properties that fully determine the $q$-matroid. 

\begin{proposition}\cite[Thm 48]{BCJ21}\label{flat $q$-matroid}
Let $\mcM = (E, \rho_{\mcM})$ be a $q$-matroid and $\F_{\mcM}$ be its collection of flats. Then $\F_{\mcM}$ satisfies the following:
\begin{enumerate}
   \item[(qF1)] $E \in \F_{\mcM}$.
    \item[(qF2)] If $F_1, F_2 \in \F_{\mcM}$ then $F_1 \cap F_2 \in \F_{\mcM}$.
   \item[(qF3)] Let $F \in \F_{\mcM}$ and $e \notin F$, then there exists a unique $F' \in \F_{\mcM}$ covering $F$ such that $F + \langle e \rangle \leq F'$.
\end{enumerate}
Furthermore, $\rho_{\mcM}$ is uniquely determined by $\F_{\mcM}$ and $\rho_{\mcM}(V) = h(\cl_{\mcM}(V))$ for all $V \leq E$, where $h(F)$ denotes the height of $F$ in the lattice $\F_{\mcM}$. Thus, we may denote this $q$-matroid as $\mcM = (E, \F_{\mcM})$.
\end{proposition}

For both matroids and $q$-matroids, there exists a notion of duality, defined respectively with complements of sets and orthogonal spaces.

\begin{definition}\label{dual matroid}
Let $M = (S, r)$ be a matroid. The \emph{dual matroid} $M^* = (S, r^*)$ is defined via the rank function 
$$r^*(A) = |A| - r(S) + r(S - A).$$
\end{definition}

Duality for $q$-matroids depends on a choice of non-degenerate symmetric bilinear form (NSBF). 
Let $E$ be a vector space over $\FF_q$ and $\langle \cdot , \cdot \rangle : E \times E \rightarrow \FF_q$ be a NSBF.  The orthogonal space of $V \leq E$ w.r.t $\langle \cdot , \cdot \rangle$ is the space 
$V^{\perp} := \{ w \in E \, : \, \langle w, v \rangle = 0 \, \textup{ for all } v \in V\}.$

\begin{definition}\label{dual $q$-matroid}
Let $\mcM = (E, \rho)$ be a $q$-matroid and $\langle \cdot , \cdot \rangle$ be a NSBF on $E$. The \emph{dual $q$-matroid} $\mcM^* = (E, \rho^*)$,  w.r.t the chosen NSBF, is defined via the rank function
$$\rho^*(V) = \dim(V) - \rho(E) + \rho(V^{\perp}).$$
\end{definition}

It was shown in \cite[Thm 2.8]{GLJ21qpolymatroids} that given two NSBFs $\langle \cdot , \cdot \rangle_1$, $\langle \cdot , \cdot \rangle_2$, the respective dual $q$-matroids $\mcM^{*_1}$ and $\mcM^{*_2}$ of $\mcM$ are equivalent. For both matroids and $q$-matroids, an element $v$ of the groundset, respectively groundspace, is a \emph{coloop} if $e$ is loop in the dual matroid, respectively dual $q$-matroid.

We now define the operations of deletion and contraction for matroids and $q$-matroids.

\begin{definition}
Let $M = (S,r)$ be a matroid and let $A \subseteq S$.
\begin{itemize}
    \item The matroid $M\setminus A = (S - A, r_{M\setminus A}),$ where $r_{M\setminus A}(B) = r(B)$ for all $B \subseteq S-A$, is called the \emph{deletion} of $A$ from $M$. 
    \item The matroid $M/A = (S -  A, r_{M/A}),$  where $r_{M/A}(B) = r(B \cup A) - r(A)$ for all $B \subseteq S-A$, is called the \emph{contraction} of $A$ from $M$.
\end{itemize}
\end{definition}

The following well-known facts about the deletion and contraction of matroids will be needed. Refer  to \cite[Prop 3.1.25]{oxley} for a proof.

\begin{proposition}\label{del cont commute}
Let $M = (S, r)$ be a matroid. Let $A, B \subseteq S$ be disjoint sets. Then 
\begin{itemize}
    \item $(M \setminus A) \setminus B = M \setminus (A \cup B) = (M \setminus B) \setminus A,$
    \item $(M/A)/B = M/(A \cup B) = (M / B)/ A,$
    \item $(M \setminus A) / B = (M / B) \setminus A.$
\end{itemize}
\end{proposition}

To avoid the surplus of parenthesis, we omit them if there is no risk of confusion. 


At this point we make a brief comment about notation. The notation $\setminus$  always denotes the deletion operation and set exclusion is denoted by  the $-$ sign. However, the notation $/$ is used to denote both the contraction of ($q$-)matroids and quotient space (i.e $E/V$). The reader should therefore use context in order to differentiate between the latter two.

For $q$-matroids, the operations of deletion and contraction are defined in an analogous way.
\begin{definition}
Let $\mcM = (E, \rho)$ be a $q$-matroid, let $V \leq E$ and fix a NSBF on $E$. Furthermore let $\pi : E \rightarrow E/V$ be the canonical projection.
\begin{itemize}
    \item The $q$-matroid $M \setminus V = (V^{\perp}, \rho_{M\setminus V})$, where $\rho_{M\setminus V}(W) = \rho(W)$ for all $W \leq V^{\perp}$ is called the \emph{deletion} of $V$ from $\mcM$.
    \item The $q$-matroid $M/V = (E/V, \rho_{M/V})$, where $\rho_{M/V}(W) = \rho(\pi^{-1}(W)) - \rho(V)$ for all $W \leq E/W$, is called the \emph{contraction} of $V$ from $\mcM$.
\end{itemize}
\end{definition}

It is worth noting that for both matroids and $q$-matroids, deletion and contraction are dual operations, i.e. $\mcM^* \setminus V \cong (\mcM / V)^*$ (equality rather than equivalence holds for matroids only). A proof of this fact for $q$-matroids can be found in \cite[Thm 5.3]{GLJ21qpolymatroids} and in \cite[Sect. 3]{oxley} for matroids.
A matroid $N$ (resp.\ $q$-matroid $\mcN$) is a \emph{minor} of $M$ (resp.\ $\mcM$) if it can be obtained from $M$ (resp.\ $\mcM$) by a sequence of deletion and contraction. 

For both matroids and $q$-matroids, the flats of a contraction can be characterized in terms of the flats of the original ($q$-)matroid.  

\begin{proposition}\label{contraction loop}
Let $M = (S, r_M)$ be a matroid, $\mcM = (E, \rho_{\mcM})$ be a $q$-matroid and $\F_M$, $\F_{\mcM}$ their respective lattice of flats. Let $A \subseteq S$,  $V \leq E$ and consider $M/A$ and $\mcM/V$. Then 
\begin{itemize}
    \item[(1)] $\F_{M/A} = \{F \subseteq S - A \, : \, F \cup A \in \F_M\}$,
    \item[(2)] $\F_{\mcM/V} = \{F \, : \, \pi^{-1}(F)  \in \F_{\mcM}\}$, where $\pi : E \rightarrow E/V$.
\end{itemize}
Furthermore $A$ (resp.\ $V$) is a flat of $M$ (resp.\ $\mcM$) if and only if $M/A$ (resp.\ $\mcM/V$) is loopless.
\end{proposition}

\begin{proof}
(1) is shown in \cite[Prop 3.3.7]{oxley}. For (2), first let $F \in \F_{\mcM/V}$ and consider the space $W: = \pi^{-1}(F) \leq E$. Let $x \notin W$. Then
$\rho_{\mcM}(W \oplus \langle x \rangle) = \rho_{\mcM/V}( F \oplus \langle \pi(x) \rangle) + \rho_{\mcM}(V) > \rho_{\mcM/V} (F) + \rho_{\mcM}(V) = \rho_{\mcM}(W)$, where the inequality holds because $F \in \F_{\mcM/V}$ and $\pi(x) \notin F$. Since this is true for all $x \notin W$ then $W \in \F_{\mcM}$.

Now let $F \leq E/V$ such that $\pi^{-1}(F) \in \F_{\mcM}$. Let $\langle x \rangle \leq F/V$ such that $x \notin F$. Then $\rho_{\mcM/V}(F \oplus \langle x \rangle) = \rho_{\mcM}(\pi^{-1}(F \oplus \langle x \rangle)) - \rho_{\mcM}(V) = \rho_{\mcM}(\pi^{-1}(F) + \pi^{-1} (\langle x \rangle)) - \rho_{\mcM}(V) >   \rho_{\mcM}(\pi^{-1}(F)) - \rho_{\mcM}(V) = \rho_{\mcM/V}(F)$. Once again, since this is true for all $x \notin F$ then $F \in \F_{\mcM/V}$.

We show the second part of the statement for matroids, and note the proof for $q$-matroid is analogous to it. Consider $M/A$ with $A \in \F_M$ and let $e \in S - A$. Then $r_{M/A}(e) = r_M(e \cup A) - r_M(A) > 0$ since $A$ is a flat. Since this holds for all $e \in S-A$,  then $M/A$ is loopless. Now assume $A \notin \F_M$ then $A \subsetneq \cl_M(A)$ and let $e \in \cl_M(A) - A \subseteq S - A$. Then $r_{M/A}( e) = r_M(A \cup e) - r(A) = 0$ since $e \in \cl_M(A).$ Hence $M/A$ contains a loop. 
\end{proof}

The last matroid operation we discuss is that of the single element extension by adjoining a loop, which we refer to as loop extension. 
The loop extension will play an important role in section 4 when defining maps between matroids. 
The reader can refer to \cite[Sect.~7.2]{oxley} and \cite[Chap.~8]{RotaMatroid} for proofs and a more detailed discussion of the single element extension.

\begin{proposition}\label{loop extension}
Let $M = (S, r)$ be a matroid and $\{o_M\}$ denotes a symbol disjoint from $S$. Let $S_o := S \cup \{o_M\}$ and $r_o : 2^{S_o} \rightarrow \NN_0$ be such that $r_o(A) = r(A - \{o_M\})$, for all $A \subseteq S_o$. Then $M_o := (S_o, r_o)$ is a matroid, and $\{o_M\}$ is a loop in $M_o$. Furthermore $M_o$ is called a \emph{loop extension of $M$}.
\end{proposition}

The subscript of the added loop may be omitted if it is clear from context in which matroid the loop is contained. The next proposition relates the flats $\F_{M_o}$ and $\F_M$. Furthermore, we recall that two lattices are \emph{isomorphic} (denoted by $\cong$) if there exists an order preserving bijection between the lattices that preserves meets and joins.

\begin{proposition}\label{flat loop extension}
Let $M$ be a matroid,  $M_o$  a loop extension of $M$, and $\F_M$, $\F_{M_o}$ their respective collection of flats. Then
$$\F_{M_o} = \{ F \cup \{o\} \, : \, F \in \F_M\}.$$
and $\F_M \cong \F_{M_o}$ as lattices.
\end{proposition}

\begin{remark}\label{loop/emptyset}
Note that $M_o \setminus \{o\}  = M$. This deletion can be seen as identifying the element $\{o\}$ with the empty set of $M$, and does not change the overall structure of the matroid. \end{remark}

\section{The Projectivization Matroid}

In \cite{JPV21}, Johnsen and co-authors showed that a $q$-matroid $\mcM$ with groundspace $E$ induces a matroid $P(\mcM)$ with groundset the projective space of $E$. This induced matroid, called the projectivization matroid of $\mcM$ turns out to be an interesting object to study. In fact, it was shown in that same paper, that the projectivization preserves the flat structure of $\mcM$. It therefore becomes a useful tool when studying properties of $q$-matroids that depend only on flats.

For completeness, we reintroduce the construction of the projectivization matroid.
The following notation will be used. Given a finite dimensional vector space $E$ over $\FF_q$, let $\PE := \{ \langle v \rangle_{\FF_q} \, : v \in E - \{0\}\}$ be the \emph{projective space of $E$}. 
The map, $\hat{P} : (E - \{0\}) \rightarrow  \PE  \, , \, v \mapsto \langle v \rangle_{\FF_q}$ induces a lattice map $P: \mcL(E) \rightarrow 2^{\PE}$, where $P(\{0\}) = \emptyset$ and $P(V) = \{\hat{P}(v) \, : \, v \in V -\{0\}\} = \{P(\langle v \rangle) \, : \, v \in V- \{0\}\}$
 for $V \leq E$. 
We call the lattice map $P$  the \emph{projectivization map}. Usually, $\hat{P}$ is called the projectivization map, however for our purposes, it is more convenient to consider the projectivization as a lattice map. Note that $P$ is inclusion preserving and that $P(V \cap W) = P(V) \cap P(W)$ for all $V, W \in \mcL(E)$.
For any $S \subseteq \PE$ let $P^{-1}(S) := \{v \in E \, : \, \hat{P}(v) \in S\} = \{v \in E \, : \, P(\langle v \rangle) \in S\}$. Note that $(P^{-1} \circ P)(V) = V$ for all $V \leq E$.
Finally let $\langle S \rangle := \langle P^{-1}(S) \rangle_{\FF_q}$ for any $S \subseteq \PE$. We say $S \subseteq \PE$ \emph{contains a basis} of $E$ if $\langle S \rangle = E$.
We can now introduce the projectivization matroid.

\begin{theorem}(\cite[Def.14, Prop. 15]{JPV21})\label{$q$-matroid/matroid}
Let $\mcM = (E, \rho)$ be a $q$-matroid and let $r: 2^{\PE} \rightarrow \NN_0$ such that for all $S \subseteq \PE, $
$$r(S) = \rho(\langle S \rangle).$$
Then $P(\mcM) := (\PE, r)$ is a matroid, and is called the \emph{projectivization matroid of $\mcM$}.
\end{theorem}

We now turn towards the relation between the flats of a $q$-matroid $\mcM$ and those of its projectivization matroid $P(\mcM)$. In the following result, the meet and join refers to those of the lattice of flats defined in Section 1. 

\begin{lemma}\cite[Lem.16, Prop.21]{JPV21}\label{flats properties}
Let $\mcM$ be a $q$-matroid, $P(\mcM)$ its projectivization matroid,  and $\F_{\mcM}$, $\F_{P(\mcM)}$ their respective lattice of flats. Furthermore, let $P(\F_{\mcM}):= \{P(F) \, : \, F \in \F_{\mcM}\}$. Then the following hold:
\begin{itemize}
    \item[1)] $\F_{P(\mcM)} = P(\F_{\mcM})$.
    
    \item[2)] $P(F_1 \vee F_2) = P(F_1) \vee P(F_2)$ and $P(F_1 \wedge F_2) = P(F_1) \wedge P(F_2)$,  for all $F_1, F_2 \in \F_{\mcM}$.
    
\end{itemize}

Therefore $\F_{P(\mcM)} \cong \F_{\mcM}$ as lattices.
\end{lemma}

The next result shows when a matroid with groundset $\PE$ is the projectivization matroid of a $q$-matroid with groundspace $E$.

\begin{theorem}\label{flat iso}
Let $M = (\PE, r)$ be a matroid and $\F_M$ its lattice of flats. Furthermore let $P^{-1}(\F_M) := \{P^{-1}(F) \cup \{0\} \, : \, F \in \F_M\}$. If $P^{-1}(F) \cup \{0\}$ is a subspace of $E$ for all $F \in \F_M$,  then  $\mcM = (E,  P^{-1}(\F_M))$ is $q$-matroid.
Furthermore $\F_M \cong \F_{\mcM}$.
\end{theorem}

\begin{proof}
 We show $\F:= P^{-1}(\F_M)$ is a collection of flats of a $q$-matroid by showing it satisfies ($q$F1)-($q$F3) of Proposition \ref{flat $q$-matroid}. Throughout the proof we use the fact that $\F_M$ is the collection of flats of a matroid, and hence satisfies (F1)-(F3) of Proposition \ref{flats matroid}.
Since $\F_M$ satisfies (F1), $\PE \in \F_M$, and therefore $P^{-1}(\PE) \cup \{0\} = E \in \F$. This shows (qF1).
Let $V_1 := P^{-1}(F_1)\cup \{0\}, V_2 := P^{-1}(F_2) \cup \{0\}  \in \F$. Since $F_1, F_2 \in \F_M$ then $F_1 \cap F_2 \in \F_M$.
Furthermore, $ P(V_1 \cap V_2) = P(V_1) \cap P(V_2)  = F_1 \cap F_2 \in \F_M$. Hence $V_1 \cap V_2 = P^{-1}(F_1 \cap F_2) \cup \{0\} \in \F$, showing (qF2).

Finally for (qF3), let $V := P^{-1}(F)\cup \{0\} \in \F$ and $w \notin V$. Since $P$ is inclusion preserving $P(\langle w \rangle) \notin P(V) = F$. Hence there exists a unique flat $F' \in \F_M$ covering $F$ such that $F \cup P( \langle w \rangle) \subseteq F'$. Let $V' := P^{-1}(F') \cup \{0\}$. By definition $V' \in \F$ and since $V'$ is a subspace containing $V \cup w$ then $V \oplus \langle w \rangle \leq V'$. To show $V'$ covers $V$, assume there exists $W \in \F$ such that $V \lneq W \leq V'$. Applying $P$ and using the fact that $P$ is inclusion preserving, we get $F= P(V) \subsetneq P(W) \subseteq P(V') = F'$. However because $W \in \F$ then $P(W) \in \F_M$. But $F'$ covers $F$ hence we must have that $P(W) = F'$ and therefore $W = V'$. This implies $V'$ is a cover of $V$ and shows $\F$ is the collection of flats of a $q$-matroid. 

Finally to show $\F_M$ and $\F$ are isomorphic as lattices, note that $P(\F) = \F_M$ hence by Theorem \ref{flats properties} the isomorphism follows.
\end{proof}

We now show that the lattice of flats of the $q$-matroid $\mcM$  contracted by a flat $F$ is isomorphic to the lattice of flats of $P(\mcM)/P(F)$.

\begin{theorem}\label{contraction flat iso}
Let $\mcM$ be a $q$-matroid, $P(\mcM)$ its projectivization matroid and $\F_{\mcM}, \F_{P(\mcM)}$ their respective lattice of flats. Then $\F_{\mcM/F} \cong \F_{P(\mcM)/P(F)}$ (as lattices) for any $F \in \F_{\mcM}$. 
\end{theorem}

\begin{proof}
Throughout let $F'_1, F'_2 \in \F_{\mcM/F}$ and $V_i = \pi^{-1}(F_i')$, where $\pi : E \rightarrow E/F$. By Proposition \ref{contraction loop} and Lemma \ref{flats properties},
$F'_i \in \F_{\mcM /F} \Leftrightarrow V_i \in \F_{\mcM} \Leftrightarrow P(V_i) \in \F_{P(\mcM)} \Leftrightarrow P(V_i) - P(F) \in \F_{P(\mcM) / P(F)}$. Furthermore, $F_1' = F_2' \Leftrightarrow V_1 = V_2 \Leftrightarrow P(V_1) - P(F) = P(V_2) - P(F)$. Hence there is a one-to-one correspondence between $\F_{\mcM/F}$ and $\F_{P(\mcM)/P(F)}$ described by the map $\psi:\F_{\mcM/F} \rightarrow \F_{P(\mcM)/P(F)}$, where $\psi(F'_i) = P(V_i) - P(F)$. Since the lattices of flats are finite, to show $\psi$ is a lattice isomorphism, we need only to show $\psi$ preserves meets. Recall that the meet of flats in either lattice is the intersection of the flats.
\begin{align*} 
\psi(F_1' \cap F_2')   &= P(V_1 \cap V_2) - P(F) \\
                        &= (P(V_1) \cap P(V_2)) - P(F) \\
                        &= (P(V_1) - P(F)) \cap (P(V_2) - P(F))\\
                        &= \psi(F_1') \cap \psi(F_2'),
\end{align*}
which completes the proof
\end{proof}

The next few properties about projectivization matroids, although not difficult to prove, will be useful in following sections.

\begin{proposition}\label{loop/qloop}
Let $\mcM = (E, \rho)$ be a $q$-matroid and $P(\mcM) = (\PE, r)$ its projectivization matroid. Then $\mcM$ contains a loop if and only if $P(\mcM)$ contains a loop.
\end{proposition}

\begin{proof}
Let  $\langle e \rangle \leq E$ be a 1-dimensional subspace. By definition, $r(P(\langle e \rangle)) = \rho (\langle e \rangle)$. Hence $\langle e \rangle$ is a loop in $\mcM$ if and only if $P(\langle e \rangle)$ is a loop in $P(\mcM)$.
\end{proof}

\begin{proposition}\label{spanning set project}
Let $\mcM = (E, \rho)$ be a $q$-matroid and $P(\mcM)= (\PE, r)$ its projectivization matroid. Let $A \subseteq \PE$ be such that $A$ contains a basis of $E$. Then $r(A) = r(\PE)$. 
\end{proposition}

\begin{proof}
Since $A$ contains a basis of $E$ then $\langle A \rangle = E$. Hence $r(A) = \rho(\langle A \rangle) = \rho(E) = r(\PE)$.
\end{proof}

We conclude the section by studying the relation between minors of a $q$-matroid and minors of its projectivization matroid. To do so we introduce the following notation. 

\begin{notation}\label{notation}
Let $V \leq E$. 
\begin{itemize}
    \item $\mcQ_V := \{ \langle w \rangle \in \PE \, : \, \langle w \rangle \not \leq V \} = \PE - P(V)$. 
    \item $\mcQ_{V}^{*A} := \mcQ_{V}- A$  for $A \subseteq \mcQ_{V}$.
\end{itemize}
\end{notation}

Note that $\PE - \mcQ_V = \PP V$. 
Furthermore for spaces $W,V \leq E$ such that $W \oplus V = E$, every element in $E/W$ can be written as $v + W$ for a unique $v \in V$. Thus the map $\psi: E/W \rightarrow V$, $v + W \mapsto v$ is a well-defined vector space isomorphism and induces a bijection on projective spaces. By slight abuse of notation we use $\psi$ as both the vector space isomorphism and the projective space bijection.
 It can then easily be shown that  $\langle \psi(A) \rangle = \psi(\langle A \rangle)$ for all $A \subseteq \PP(E/W)$.

\begin{theorem}\label{contraction proj}
Let $\mcM = (E, \rho)$ be a $q$-matroid and let $W,V \leq E$ such that $W \oplus V = E$. Let $V^{\perp}$ be the orthogonal space of $V$ w.r.t a fix NSBF. Furthermore let $S = \{\langle w_1 \rangle , \cdots , \langle w_t \rangle \} \subseteq \PE$ such that $\{w_1, \cdots, w_t\}$ is a basis of $W$.Then 
$$P(\mcM/W) \cong (P(\mcM) / S )\setminus \mcQ_{V}^{*S}  $$
$$P(\mcM \setminus V^{\perp}) \cong P(\mcM) \setminus \mcQ_{V}.$$
\end{theorem}

\begin{proof}
Let $N := P(\mcM)/S \setminus \mcQ_{V}^{*S}$. Note $N$ has groundset $\PE - \mcQ_{V} = \PP V$ whereas $P(\mcM/W)$ has groundset $\PP(E/W)$. Let $\psi : \PP(E/W) \rightarrow \PP V$ be the bijection described previously.
To show $N \cong P(M/W)$, we must show $r_{P(\mcM/W)}(A) = r_N(\psi(A))$ for all $A \subseteq \PP(E/W)$. 
Let $\pi : E \rightarrow E/W$ be the canonical projection. Since $S$ is a basis of $W$ then $\pi^{-1}(\langle A \rangle) = \langle \psi(A) \rangle + W = \langle \psi(A) \cup S \rangle.$ Furthermore by Theorem \ref{$q$-matroid/matroid}, $r_{P(\mcM)}(S) = \rho(W)$. Hence we get:
\begin{align*}
r_{P(\mcM/W)}(A) &= \rho_{\mcM/W}(\langle A \rangle)\\
                                    &= \rho_{\mcM}(\pi^{-1}(\langle A\rangle)) - \rho_{\mcM}(W)\\
                                    &= \rho_{\mcM}(\langle \psi(A) \cup S \rangle) - \rho_{\mcM}(\langle S \rangle)\\
                                    &= r_{P(\mcM)}( \psi(A) \cup S) - r_{P(\mcM)}(S) \\
                                    &= r_{P(\mcM)/ S}(\psi(A))\\
                                    &= r_{N}(\psi(A)),
\end{align*}

where the last equality holds because $\psi(A) \subseteq \PP V$ which is the groundset of $N$.

Moving on to the second equivalence. Both matroid $P(\mcM \setminus V^{\perp})$ and $P(\mcM) \setminus \mcQ_{V}$ have groundset $\PP V$. Hence we need only to show that the rank functions of both matroids are equal. 
Let $A \subseteq \PP V$. 
\begin{align*}
    r_{P(\mcM \setminus V^{\perp})}(A) &= \rho_{\mcM \setminus V^{\perp}}(\langle A \rangle)\\
    &= \rho_{\mcM}(\langle A \rangle) \\
    &= r_{P(\mcM)}(A) \\
    &= r_{P(\mcM) \setminus \mcQ_{V}}(A),
\end{align*}
where the last equality follows because $A \subseteq \PP V$ which is  the groundset of $P(\mcM)\setminus\mcQ_V$.

\end{proof}

\section{Maps of Matroids and $q$-Matroids}

Maps between matroids were first introduced to study matroids from a category theory approach. The reader may refer to \cite{Heunen2018TheCO,RotaMatroid} for more details. Maps between $q$-matroids were defined for the same purposes, but were only recently introduced in \cite{GLJ21}. In this section we focus on the relation between maps of $q$-matroids and maps of matroids, and show that the projectivization map is a functor from categories of $q$-matroids to categories of matroids. 
This, in turn, provides a new approach to study maps between $q$-matroids.
For both matroids and $q$-matroids, there exists two main types of maps that preserve the matroid structure: weak and strong maps. To avoid confusion between maps of matroids and maps of $q$-matroids, we use the terms weak and strong maps for the former, and, $q$-weak and $q$-strong maps for the latter. We recall the definitions of weak and strong maps between matroids and some of their properties. 

The reader should note that the loop extension (Definition \ref{loop extension}) is needed to define maps between matroids. As the next definition will show a map between matroids is a map defined on the groundset of the loop extension matroids.
By remark \ref{loop/emptyset}, the added loop can be seen as an element representing the empty set of the matroid. Hence, mapping an element to the added loop of the codomain can be seen as mapping an element to the empty set.

\begin{definition}\label{Strong Map 1}
Let $M = (S, r_M)$ and $N = (T , r_N)$ be matroids and $M_o$, $N_o$ be their respective loop extension matroids. A map $\sigma : M \rightarrow N$ is a map between the groundset of the loop extension matroids, i.e. $\sigma : S_o \rightarrow T_o$, such that $\sigma(o_M) = o_N$. Furthermore $\sigma$ is said to be:

\begin{itemize}
    \item \emph{weak} if  $r_{N_o}(\sigma (A)) \leq r_{M_o}(A)$ for all $A \subseteq S_o$.
    
    \item \emph{strong} if $\sigma^{-1}(F) \in \F_{M_o}$ for all $F \in \F_{N_o}$.
\end{itemize}
\end{definition}

It is well known (see \cite[Chap.~8, Lem.~8.1.7]{RotaMatroid}) that strong maps are weak maps. Furthermore a map $\sigma: M \rightarrow N$ induces a map $\sigma^{\#} : \F_{M_o} \rightarrow \F_{N_o}$, where $\sigma^{\#}(F) = \cl_{N_o}(\sigma(F))$ for all $F \in \F_{M_o}$. Using Proposition \ref{flat loop extension}, one can alternatively define $\sigma^{\#} : \F_M \rightarrow \F_N$. As the following theorem shows, the induced map $\sigma^{\#}$ provides an alternative definition for strong maps.

\begin{theorem}\label{Strong Map 2}(\cite[Prop 8.1.3]{RotaMatroid})

A map $\sigma: M \rightarrow N$ is a strong map  if and only if the following hold:

\begin{itemize}
    \item[(1)] for all $F_1, F_2 \in \F_M$,  
    $$\sigma^{\#}(F_1 \vee F_2) = \sigma^{\#}(F_1) \vee \sigma^{\#}(F_2)$$
    
    \item[(2)] $\sigma^{\#}$ sends atoms of $\F_{M}$ to atoms or to the zero of  $\F_N$. 
\end{itemize}
\end{theorem}

The main result of this section is the analogue of Theorem \ref{Strong Map 2} for $q$-matroids.
We turn to the definitions of maps between $q$-matroids, as introduced in \cite{GLJ21}. Similarly to matroids, maps between $q$-matroids are maps between groundspaces that send subspaces to subspaces. 

\begin{definition}\label{L-maps}
A map $\sigma : E_1 \rightarrow E_2$ is an \emph{$\mcL$-map} if $\sigma(V) \in \mcL(E_2)$ for all $V \in \mcL(E_1)$. An $\mcL$-map $\sigma$ induces a map $\sigma_{\mcL} : \mcL(E_1) \rightarrow \mcL(E_2)$. Two $\mcL$-maps $\sigma, \psi$ from $E_1$ to $E_2$ are \emph{$\mcL$-equivalent}, denoted by $\sigma \sim_{\mcL} \psi$, if $\sigma_{\mcL} = \psi_{\mcL}$.
\end{definition}

We consider the following two types of maps between $q$-matroids.

\begin{definition}
Let $\mcM = (E_1, \rho_{\mcM})$  and $\mcN = (E_2, \rho_{\mcN})$ be $q$-matroids. A map $\sigma: \mcM \rightarrow \mcN$ is an $\mcL$-map between the groundspaces of $\mcM$ and $\mcN$, i.e $\sigma: E_1 \rightarrow E_2$. Furthermore $\sigma$ is said to be
\begin{itemize}
    \item[(a)]  \emph{$q$-weak} if $\rho_{\mcN}(\sigma(V)) \leq \rho_{\mcM}(V)$ for all $V \leq E_1$.
    \item[(b)] \emph{$q$-strong}  if $\sigma^{-1}(F) \in \F_{\mcM}$ for all $F \in \F_{\mcN}$.
\end{itemize}
\end{definition}

To study the relation between maps of matroids and maps of $q$-matroids we need the following notation. 
Given a vector space $E$, define the \emph{extended projective space of $E$} as $\PoE = \PE \cup \{o\}$, where $\{o\}$ is an arbitrary element disjoint from $\PE$.  Let $P_o : E \rightarrow \PoE$, where $P_o(0) = o$ and $P_o(v) = \hat{P}(v)$ for $v \neq 0$ and $\hat{P}: E - \{0\} \rightarrow \PE$ is as introduced in the previous section. We call $P_o$ the \emph{extended projectivization map}. 
Note, unlike the projectivization map, we do not consider $P_o$ as lattice map but as a map between a vector space to its extended projective space.
Given a $q$-matroid $\mcM = (E, \rho)$ and the loop extension of its projectivization matroid $P(\mcM)_o = (\PoE, r_o)$, the map $P_o$ can be viewed as a map between the groundspace $E$ to the groundset $\PoE$ such that $\rho(V) = r_o(P_o(V))$ for all $V \leq E$. Furthermore for any $A \subseteq \PoE$ let $\langle A \rangle := \langle P_o^{-1}(A) \rangle_{\FF_q}$. It can easily be shown that $r_o(A) = \rho(\langle A \rangle )$.

Recall from Definition \ref{L-maps} that an $\mcL$-map $\sigma : E_1 \rightarrow E_2$ induces a map on the lattices of subspaces $\sigma_{\mcL} : \mcL(E_1) \rightarrow \mcL(E_2)$. By restricting $\sigma_{\mcL}$ to the $1$-dimensional spaces and the $0$ of $E_1$, $\sigma_{\mcL}$ can be viewed as map between the extended projective spaces $\PoE_1$ and $\PoE_2$, i.e $\sigma_{\mcL} : \PoE_1 \rightarrow \PoE_2$. As the next proposition shows, $\sigma$ and $\sigma_{\mcL} $ commute with the extended projectivization map.  

\begin{proposition}\label{maps commuting}
Let $\sigma : E_1 \rightarrow E_2$ be an $\mcL$-map, $\sigma_{\mcL} : \PoE_1 \rightarrow \PoE_2$ its induced map on the extended projective spaces, and $P_o : E_i \rightarrow \PoE_i$ be the extended projectivization map. Then 
$$P_o \circ \sigma = \sigma_{\mcL} \circ P_o.$$
\end{proposition}

\begin{proof}
Let $v \in E_1$. Since $\sigma$ is an $\mcL$-map then $\langle \sigma(v) \rangle = \sigma(\langle v \rangle) =  \sigma_{\mcL}(\langle v \rangle)$.
But note $\langle \sigma( v ) \rangle = P_o(\sigma(v))$ and $\langle v \rangle = P_o( v )$. Hence the wanted equality follows.
\end{proof}

We now consider the case when an $\mcL$-map $\sigma$ is a map between $q$-matroids. The induced map $\sigma_{\mcL}$ between the extended projective spaces turns out to be a map between projectivization matroids. Furthermore $\sigma$ being $q$-weak or $q$-strong is fully determined by whether $\sigma_{\mcL}$ is weak or strong, and vice versa.

\begin{theorem}\label{maps/$q$-maps}
Let $\mcM = (E_1, \rho_{\mcM})$, $\mcN = (E_2, \rho_{\mcN})$ be $q$-matroids and $P(\mcM) = (\PE_1, r_{P(\mcM)})$, $P(\mcN) = (\PE_2, r_{P(\mcN)})$ be their projectivization matroids. Let $\sigma : \mcM \rightarrow \mcN$ be an $\mcL$-map. Then $\sigma_{\mcL}: P(\mcM) \rightarrow P(\mcN)$ is a map between the projectivization matroids and the following holds:
\begin{itemize}
    \item $\sigma$  is $q$-weak if and only if $\sigma_{\mcL}$ is weak
    
    \item $\sigma$ is $q$-strong if and only if $\sigma_{\mcL}$ is strong. 
\end{itemize}
\end{theorem}

\begin{proof}
To start note that $\sigma_{\mcL}$ is a map between the groundsets of the loop extension matroid $P(\mcM)_o$ and $P(\mcN)_o$ with $\sigma_{\mcL}(o_{P(\mcM)}) = o_{P(\mcN)}$. Thus $\sigma_{\mcL} : P(\mcM) \rightarrow P(\mcN)$ is well defined.

We first prove $\sigma$ is $q$-weak if and only if $\sigma_{\mcL}$ is weak. Assume $\sigma$ is weak. Let $A \subseteq \PoE_1$ and $V := \langle A \rangle$. Both $P_o$ and $\sigma$ are inclusion preserving, hence, $(P_o \circ \sigma) (P_o^{-1}(A)) \subseteq (P_o \circ \sigma)(V)$.
Using Proposition \ref{maps commuting} on the first term of the previous inclusion gives us $(\sigma_{\mcL} \circ P_o)(P_o^{-1}(A)) = \sigma_{\mcL}(A) \subseteq (P_o \circ \sigma)(V)$. Furthermore, by the monotonicity property of the rank functions and because $\sigma$ is weak, we get
$$r_{P(\mcN)_o}(\sigma_{\mcL}(A)) \leq r_{P(\mcN)_o}(( P_o \circ \sigma)(V))) = \rho_{\mcN}(\sigma(V)) \leq \rho_{\mcM}(V) = r_{P(\mcM)_o}(A).$$
Because $A \subseteq \PoE_1$ was arbitrarily chosen, then $\sigma_{\mcL}$ is weak. 

Now assume $\sigma_{\mcL}$ is weak. Let $V \leq E_1$ and recall $\rho_{\mcM}(V) = r_{P(\mcM)_o}(P_o(V))$. Since $\sigma_{\mcL}$ is weak  $r_{P(\mcN)_o}( (\sigma_{\mcL} \circ P_o)(V)) \leq r_{P(\mcM)_o}(P_o(V))$. Hence by Proposition \ref{maps commuting}, $r_{P(\mcN)_o}((P_o \circ \sigma)(V)) \leq r_{P(\mcM)_o}(P_o(V))$. This implies  $\rho_{\mcN}(\sigma(V)) \leq \rho_{\mcM}(V)$ and shows $\sigma$ is $q$-weak. 

\vspace{.2cm}

We now show that $\sigma$ is $q$-strong if and only if $\sigma_{\mcL}$ is strong. 
From Proposition \ref{flat loop extension} and Lemma \ref{flats properties}, $F \in \F_{\mcM} \Leftrightarrow P(F) \in \F_{P(\mcM)} \Leftrightarrow P(F) \cup \{o\} \in \F_{P(\mcM)_o}$. A similar chain of equivalence holds for $\F_{\mcN}$ and $\F_{P(\mcN)_o}$. Furthermore all flats of $P(\mcN)_o$ are of the form  $P_o(F) = P(F) \cup \{o\} $ for some flat in $\mcN$. 
Therefore $\sigma_{\mcL}$ is strong iff $\sigma_{\mcL}^{-1} (P_o(F)) \in \F_{P(\mcM)_o}$ for all $P_o(F) \in \F_{P_o(\mcN)}$ iff $ (\sigma_{\mcL} \circ P_o)^{-1}(P_o(F)) \in \F_{\mcM}$ for all $P_o(F) \in \F_{P_o(\mcN)}$ iff $(P_o \circ \sigma)^{-1}(P_o(F)) = \sigma^{-1}(F) \in \F_{\mcM}$ for all $F \in \F_{\mcN}$ iff $\sigma$ is $q$-strong.

\end{proof}

From the above theorem it can easily be seen that the projectivization map is a functor from the category of $q$-matroids with $q$-weak (resp.\ $q$-strong) map to the category of matroids with weak (resp.\ strong) maps. 
We now turn towards a proof of the analogue of Theorem \ref{Strong Map 2}. We first define the analogue of the map $\sigma^{\#}$. 

\begin{definition}\label{$q$-flat map}
Let $\mcM$ and $\mcN$ be $q$-matroids with respective groundspaces $E_1, E_2$ and  $\sigma : \mcM \rightarrow \mcN$ be an $\mcL$-map. Define $\sigma^{\#}: \F_{\mcM} \rightarrow \F_{\mcN}$ such that
 $$\sigma^{\#}(F) = \cl_{\mcN}(\sigma(F)).$$
\end{definition}

The next useful Lemma shows that the induced maps $\sigma^{\#}$ and $\sigma_{\mcL}^{\#}$ commute with the extended projectivization map. 

\begin{lemma}\label{commute flats}
Let $\mcM, \mcN$ be $q$-matroids, $\F_{\mcM} , \F_{\mcN}$ their lattices of flats and $P(\mcM), P(\mcN)$ their projectivization matroids. Furthermore let $\sigma:\mcM \rightarrow \mcN$ be an $\mcL$-map, $\sigma_{\mcL} : P(\mcM) \rightarrow P(\mcN)$ its induced map and $P_o: E_i \rightarrow \PoE_i$  the extended projectivization map.
Then 
$$P_o \circ \sigma^{\#} = \sigma_{\mcL}^{\#} \circ P_o$$
\end{lemma}

\begin{proof}
First recall, $F \in \F_{\mcM} \Leftrightarrow P_o(F) \in \F_{P(\mcM)_o}$. Let $F \in \F_M$, then $\sigma (F) \subseteq \sigma^{\#}(F)$ and since $P_o$ is inclusion preserving  $(P_o \circ \sigma)(F) \subseteq (P_o \circ \sigma^{\#})(F)$.
By Proposition \ref{maps commuting}, the above containment implies $(\sigma_{\mcL} \circ P_o)(F) \subseteq (P_o \circ \sigma^{\#})(F)$. 
Applying the closure operator of $P(\mcN)_o$, we get
$$(\sigma_{\mcL}^{\#} \circ P_o)(F) = \cl_{P(\mcN)_o}((\sigma_{\mcL} \circ P_o)(F)) \subseteq \cl_{P(\mcN)_o}((P_o \circ \sigma^{\#})(F)) = (P_o \circ \sigma^{\#})(F),$$
where the final equality holds because $\sigma^{\#}(F) \in \F_{\mcN}$ and therefore $(P_o \circ \sigma^{\#})(F) \in \F_{P(\mcN)_o}$. 
 Assume, for sake of contradiction, that $(\sigma_{\mcL}^{\#} \circ P_o)(F) \subsetneq (P_o \circ \sigma^{\#}) (F)$. Let $F' := (\sigma_{\mcL}^{\#} \circ P_o)(F)$. Then
$$(\sigma_{\mcL} \circ P_o)(F) \subseteq F' \subsetneq (P_o \circ \sigma^{\#}) (F).$$

By considering their preimage under $P_o$ and because $\sigma_{\mcL} \circ P_o = P_o \circ \sigma$,  we get
$$\sigma(F) \subseteq P_o^{-1}(F') \subsetneq \sigma^{\#}(F).$$

However since $F' \in \F_{P(\mcN)_o}$ then $P_o^{-1}(F') \in \F_{\mcN}$. Therefore $P_o^{-1}(F')$ must contain $\cl_{\mcN}(\sigma(F)) = \sigma^{\#}(F)$, a contradiction. Hence $$(\sigma_{\mcL}^{\#} \circ P_o)(F) = (P_o \circ \sigma^{\#}) (F)$$
\end{proof}

In the statement of the previous Lemma, one can replace the extended projectivization map $P_o$ by the projectivization map $P$ introduced in the previous section. In fact, as previously remarked, the map $\sigma_{\mcL}^{\#}$ can be considered as map between the lattices of flats  $\F_{P(\mcM)}$ to $\F_{P(\mcN)}$. Furthermore the projectivization map can also be restricted to a map between the lattice of flats of a $q$-matroid and its projectivization matroid. In the following Lemma, $P$ refers to the projectivization map restricted to the lattice of flats $\F_{\mcM}$ and $\F_{\mcN}$.

\begin{lemma}\label{commute flats 2}
Let the data be as in Lemma \ref{commute flats}, and let $P$ be the projectivization map on the lattices of flats $\F_{\mcM}$ and $\F_{\mcN}$. Then 
$$P \circ \sigma^{\#} = \sigma_{\mcL}^{\#} \circ P.$$
\end{lemma}

\begin{proof}
Recall $\F_{P(\mcM)} = \{ F' - \{o\} \, : \, F' \in P(\mcM)_o\} = \{P_o(F) - \{o\} \, : \, F \in \F_{\mcM}\} = \{P(F) \, : \, F \in \F_{\mcM}\}$ and that the same holds for $P(\mcN)$. From the above chain of equality and Lemma \ref{commute flats}, equality follows straightforwardly.
\end{proof}
We now state and show the analogue of Theorem \ref{Strong Map 2} for $q$-strong maps. 
\begin{theorem}\label{$q$-strong map 2}

Let $\mcM$, $\mcN$ be $q$-matroids. An $\mcL$-map $\sigma: \mcM \rightarrow \mcN$ is a $q$-strong map  if and only if the following holds:

\begin{itemize}
    \item[(1)] for all $F_1, F_2 \in \F_{\mcM}$,  
    $$\sigma^{\#}(F_1 \vee F_2) = \sigma^{\#}(F_1) \vee \sigma^{\#}(F_2)$$
    
    \item[(2)] $\sigma^{\#}$ sends of $\F_{\mcM}$ atoms to atoms or to the zero of $\F_{\mcN}$. 
\end{itemize}
\end{theorem}

\begin{proof}

($\Rightarrow$) Let $\sigma : \mcM \rightarrow \mcN$ be a $q$-strong map, which implies by Theorem \ref{maps/$q$-maps} that $\sigma_{\mcL}: P(\mcM) \rightarrow P(\mcN)$ is a strong map. 
By Lemma \ref{flats properties} (1), $F \in \F_{\mcM} \Leftrightarrow P(F) \in \F_{P(\mcM)}$. Furthermore, from Theorem \ref{Strong Map 2} we obtain  $\sigma_{\mcL}^{\#}(P(F_1) \vee P(F_2)) = \sigma_{\mcL}^{\#}(P(F_1)) \vee \sigma_{\mcL}^{\#}(P(F_2))$ for all $F_1, F_2 \in \F_{P(\mcM)}$. 
 By Lemma \ref{flats properties} (2), $P(F_1) \vee P(F_2) = P(F_1 \vee F_2)$, hence $\sigma_{\mcL}^{\#}(P(F_1 \vee F_2)) = \sigma_{\mcL}^{\#}(P(F_1)) \vee \sigma_{\mcL}^{\#}(P(F_2))$.
Applying Lemma \ref{commute flats 2} on the above equalities gives us 
\begin{align*}
    (P \circ \sigma^{\#})(F_1 \vee F_2) &= (P \circ \sigma^{\#})(F_1) \vee (P \circ \sigma^{\#})(F_2)\\
    &= P ( \sigma^{\#}(F_1) \vee \sigma^{\#}(F_2)).
\end{align*}
Finally since $P$ is an isomorphism on the lattice of flat, the above equality implies
$$\sigma^{\#}(F_1 \vee F_2) = \sigma^{\#}(F_1) \vee \sigma^{\#}(F_2),$$
which shows $\sigma$ satisfies property (1) for all $F_1, F_2 \in \F_{\mcM}$.

To show $\sigma$ satisfies property (2), let $F \in \F_{\mcM}$ be an atom. Since $P$ is a lattice isomorphism then $P(F)$ is an atom of $\F_{P(\mcM)}$. Moreover $\sigma_{\mcL}$ is a strong map, hence by Theorem \ref{Strong Map 2}, $(\sigma_{\mcL}^{\#} \circ P)(F)$ must be an atom or the zero of $\F_{P(\mcN)}$. But by Lemma \ref{commute flats 2}, $(\sigma_{\mcL}^{\#} \circ P)(F) = (P \circ \sigma^{\#})(F)$, which implies $\sigma^{\#}(F)$ must be an atom or the zero of $\F_{\mcN}$ because, once again, $P$ is a lattice isomorphism. This concludes  that $\sigma^{\#}$ satisfies the wanting properties.

($\Leftarrow$) Let $\sigma^{\#}$ satisfy properties (1) and (2). 
We show that $\sigma$ is a $q$-strong map by showing that $\sigma_{\mcL}$ is a strong map. To do so we show that $\sigma_{\mcL}^{\#}$ satisfies Proposition \ref{Strong Map 2}.

Let $P(F_1), P(F_2) \in \F_{P(\mcM)}$. By Lemma \ref{flats properties} (2) $\sigma_{\mcL}^{\#}(P(F_1) \vee P(F_2)) = \sigma_{\mcL}^{\#}(P(F_1 \vee F_2)).$ Using Lemma \ref{commute flats 2} and the fact that $\sigma^{\#}$ satisfies property (1), we get
\begin{align*}
    (\sigma_{\mcL}^{\#} \circ P)(F_1 \vee F_2) &= (P \circ \sigma^{\#})(F_1 \vee F_2)\\
                                        &= P ( \sigma^{\#}(F_1) \vee \sigma^{\#}(F_2))\\
                                        &= (P \circ \sigma^{\#})(F_1) \vee (P \circ \sigma^{\#})(F_2)\\
                                        &= (\sigma_{\mcL}^{\#} \circ P)(F_1) \vee (\sigma_{\mcL}^{\#} \circ P)(F_2),
\end{align*}
where the second to last equality follows from Lemma \ref{flats properties} (2). Hence $\sigma_{\mcL}^{\#}$ satisfies property (1) of Prop \ref{Strong Map 2}.

Let  $P(F) \in \F_{P(\mcM)}$ be an atom which, since $P$ is a lattice isomorphism, implies $F$ is an atom of $\F_{\mcM}$. Once again  we use $(\sigma_{\mcL}^{\#} \circ P)(F) = (P \circ \sigma^{\#})(F).$
Because $\sigma^{\#}$ satisfies property (2) then $\sigma^{\#}(F)$ is an atom or the zero of $\F_{\mcN}$. Finally since $P$ is a lattice isomorphism between $\F_{\mcN}$ and $\F_{P(\mcN)}$ then $P(\sigma^{\#}(F)) = \sigma_{\mcL}^{\#}(P(F))$ must be an atom or the zero of $\F_{P(\mcN)}$ which show $\sigma_{\mcL}^{\#}$ satisfies property (2) of Proposition \ref{Strong Map 2}.
 Therefore $\sigma_{\mcL}$ is a strong map and by Theorem \ref{maps/$q$-maps} we get that $\sigma$ is $q$-strong, concluding the proof. 
\end{proof}

We conclude the section by showing that $q$-strong maps are $q$-weak maps. 

\begin{corollary}
Let $\mcM = (E_1, \rho_{\mcM})$, $\mcN = (E_2, \rho_{\mcN})$ be $q$-matroids and $\sigma : \mcM \rightarrow \mcN$ be a $q$-strong map. Then $\sigma$ is a $q$-weak map. 
\end{corollary}

\begin{proof}
Let $V \leq E_1$, such that $\rho_{\mcM}(V) = s$ and let $F := \cl_{\mcM}(V)$. Note that $V \subseteq F$ and $\rho_{\mcM}(F) = \rho_{\mcM}(V) =h(F)$, where $h$ is the height function of the lattice of flats as in Proposition \ref{flat $q$-matroid}.
Because $\F_{\mcM}$ is a geometric lattice, and $F$ has height $s$, then  $F = \bigvee_{i=1}^s a_i$ where $a_i$ are atoms of $\F_{\mcM}$.
Since $\sigma$ is a $q$-strong map, by Theorem \ref{$q$-strong map 2} we get, $\sigma^{\#}(F) = \sigma^{\#}(\bigvee_{i=1}^s a_i) = \bigvee_{i=1}^s \sigma^{\#}(a_i)$. 
Because $a_i$ are atoms of $\F_M$,by Theorem \ref{$q$-strong map 2}, $\sigma^{\#}(a_i)$ must be an atom or the zero of $\F_{\mcN}$ and that for all $1 \leq i \leq s$. Hence $\sigma^{\#}(F)$ is be the join of at most $s$-atoms, which implies $ \rho_{\mcN}(\sigma^{\#}(F))= h(\sigma^{\#}(F))  \leq s$. 
Finally, since $V \leq F$ and $\sigma$ is an $\mcL$-map, then $\sigma(V) \leq \sigma(F) \leq \sigma^{\#}(F)$, where the last containment follows from the definition of $\sigma^{\#}$. 
Therefore by the monotonicity property of $\rho_{\mcN}$ we get  $\rho_{\mcN}(\sigma(V)) \leq \rho_{\mcN}(\sigma^{\#}(F)) \leq s = \rho_{\mcM}(V)$. Since this holds true for all $V \leq E_1$, $\sigma$ must be $q$-weak.
\end{proof}

\section{The Characteristic Polynomial}
The characteristic polynomial is a useful invariant for both matroids and $q$-matroids. For the former it was intensively studied over the years, see for example \cite{welsh2010matroid, zaslavsky1987}. The latter was more recently introduced for $q$-polymatroids \cite{BCIJ21}, and was used to establish a weaker version of the Assmus-Mattson Theorem. However, in this paper, we are only interested in the characteristic polynomial of $q$-matroids.  

Before defining the characteristic polynomial, we recall the definition of the M\"obius function which will often be used throughout the section. 

\begin{definition}\label{Mobius function}
Let $(P , \leq )$ be a finite partially ordered set. The M\"obius function for $P$ is defined via the recursive formula
$$\mu_P(x,y) := \begin{cases} 1 & \textit{if } x = y, \\ - \sum_{x \leq z \lneq y} \mu_P(x,z) & \textit{if } x < y, \\
0 & \textup{otherwise.} \end{cases}$$
\end{definition}

We use the subscript of $\mu$ to distinguish between the M\"obius functions of different posets.
If the underlying poset is clear, the subscript may be omitted. 
We now define the characteristic polynomial of a matroid. 

\begin{definition}
Let $M = (S, r)$ be a matroid. The characteristic polynomial of $M$ is defined as follow:
$$\chi_M(x) = \sum_{A \subseteq S} (-1)^{|A|}x^{r(S) - r(A)}.$$
\end{definition}

It is well known that if a matroid $M$ contains a loop, its characteristic polynomial is identically 0. On the other hand if $M$ is loopless, then the characteristic polynomial of $M$ is fully determined by the lattice of flats. 
Furthermore, one can recursively define the characteristic polynomial of a matroid in terms of the characteristic polynomial of its minors. We summarize this in the following theorem. Proofs can be found in \cite[Sec.3]{welsh2010matroid} and \cite[Sec.7.1]{zaslavsky1987}.

\begin{theorem}\label{char matroid}
Let $M = (S, r)$ be a matroid and $\F$ be its lattice of flats. If $M$ contains a loop then $\chi_M(x) = 0$. If $M$ has no loops, then 
$$\chi_M(x) = \sum_{F \in \F} \mu_{\F}(0,F) x^{r(S) - r(F)}.$$ 

Furthermore for $e \in S$, 
$$\chi_M(x) = \begin{cases} \chi_{M \setminus e}(x) \chi_{M/e}(x) & \textup{if } e \textup{ is a coloop,}\\
\chi_{M\setminus e}(x) - \chi_{M/e}(x)  & \textup{otherwise.}
\end{cases}$$

\end{theorem}

Similarly to matroids, the characteristic polynomial of a $q$-matroid $\mcM$ is identically 0 if $\mcM$ contains a loop, and is fully determined by the lattice of flats otherwise. This was in fact shown by Whittle in \cite{whittle1992characteristic}, where the author generalized the result to any weighted lattice endowed with a closure operator. However for self containment purposes proofs of those facts will be included in this paper.
Furthermore, we use the projectivization matroid to find a recursive formula for the characteristic polynomial of $q$-matroids. The characteristic polynomial of a $q$-matroid is defined in the following way.

\begin{definition}\cite[Def. 22]{BCIJ21}
Let $\mcM = (E, \rho)$ be a $q$-matroid and $\mcL(E)$ be the subspace lattice of $E$. The \emph{characteristic polynomial} is defined as 
$$\chi_{\mcM}(x):= \sum_{V \leq E} \mu_{\mcL(E)}(0, V)x^{\rho(E) - \rho(V)}.$$
\end{definition}

We state a few straightforward lemmas that will be useful later on. 

\begin{lemma}\label{space of loops} 
Let $\mcM = (E, \rho)$ be a $q$-matroid, $\F_{\mcM}$ its lattice of flats, and $L := \{ e \in E \, | \, \rho(\langle e \rangle ) = 0\}$. Then $L$ is a subspace and $\rho(L) = 0$. It is called the \emph{the subspace of loops} of $\mcM$. Furthermore $L \leq F$ for all $F \in \F_{\mcM}$.
\end{lemma}

\begin{proof}
The first statement was proven in \cite[Lemma 11]{JP18}. For the second statement let $V \leq E$. By the monotonicity and submodularity properties of the rank function, $\rho(V) \leq \rho(V + L) \leq \rho(V) + \rho(L) - \rho(V \cap L) \leq \rho(V)$. Hence equality holds throughout and $L \leq \cl_{\mcM}(V)$. Since this is true for all $V \leq E$ then $L \leq F$ for all flats $F$ of $\mcM$.
\end{proof}

\begin{lemma}\label{Mobius flat decomp}
Let $\mcM= (E, \rho)$ be a $q$-matroid, $\F$ its lattice of flats and $\mcL$ the subspace lattice of $E$. Then
$$\chi_{\mcM}(x) = \sum_{F \in \F} \hspace{.1cm} \sum_{V \, : \, \cl_{\mcM}(V) = F} \mu_{\mcL}(0, V)  x^{\rho(E) - \rho(F)}.$$
\end{lemma}

\begin{proof}
Let $V \leq E$, then $\cl_{\mcM}(V) \in \F$  and $\rho(V) = \rho(\cl_{\mcM}(V))$. Hence we get
\begin{align*}
    \chi_M(x)   &= \sum_{V \leq E} \mu_{\mcL}(0, V)x^{\rho(E) - \rho(V)}\\
                &= \sum_{F \in \F}\hspace{.1cm}  \sum_{V \, : \, \cl_{\mcM}(V) = F} \mu_{\mcL}(0, V) x^{\rho(E) - \rho(F)}. \quad \qedhere
\end{align*}
\end{proof}

We can now show that if a $q$-matroid contains a loop its characteristic polynomial is identically 0, whereas if it is loopless then the characteristic polynomial is determined by the lattice of flats. 

\begin{theorem}\label{loop char 0}[see also \cite[Prop. 3.4]{whittle1992characteristic}]
Let $\mcM$ be a $q$-matroid such that $\mcM$ contains a loop then $\chi_{\mcM}(x) = 0$.
\end{theorem}

\begin{proof}
From Lemma \ref{Mobius flat decomp}, we know 
$$\chi_{\mcM}(x) = \sum_{F \in \F}  \sum_{V \, : \, \cl(V) = F} \mu_{\mcL}(0, V)  x^{\rho(E) - \rho(F)}.$$

We show that for all flat $F \in \F$, 
\begin{equation}\label{fl}
\sum_{V \, : \, \cl(V) = F} \mu_{\mcL}(0, V) = 0.
\end{equation}

We proceed by induction on the rank value of flats.
Let $F \in \F_{\mcM}$ such that $\rho(F) = 0$, i.e $F = \cl_{\mcM}(0)$. Since $\mcM$ contains a loop, $\{0\} \subsetneq F$. Hence by Definition \ref{Mobius function},
$$\sum_{V \, : \, \cl(V) = F} \mu_{\mcL}(0, V) = \sum_{0 \leq V \leq F}\mu_{\mcL}(0, V) = 0.$$

Assume (\ref{fl}) holds for all $F \in \F$ such that $\rho(F) \leq k-1$. 

Fix $F \in \F$ such that $\rho(F) = k$. Then 
\begin{align*}
   0 &=  \sum_{V \leq F} \mu_{\mcL} (0, V) \\
   &= \sum_{V \, : \, \cl(V) = F} \mu_{\mcL} (0,V) + \sum_{F' \lneq F, F' \in \F} \hspace{.1cm} \sum_{V \, : \, \cl(V) = F'} \mu_{\mcL}(0,V)\\
   &= \sum_{V \, : \, \cl(V) = F} \mu_{\mcL} (0,V),
\end{align*}
where the last equality follows by induction hypothesis. Therefore (\ref{fl}) holds, and $\chi_{\mcM}(x) = 0.$
\end{proof}

\begin{theorem}\label{char flats}[see also \cite[Thm. 3.2]{whittle1992characteristic}]
Let $\mcM = (E, \rho)$ be a loopless $q$-matroid and $\F$ its lattice of flats. Then 
$$\chi_{\mcM}(x) = \sum_{F \in \F} \mu_{\F}(0, F) x^{\rho(E) - \rho(F)}.$$
\end{theorem}

\begin{proof}
Let $\mcL = \mcL(E)$.
By Lemma \ref{Mobius flat decomp} we know 
$$\chi_{\mcM}(x) = \sum_{F \in \F} \hspace{.1cm} \sum_{V \, : \, \cl(V) = F} \mu_{\mcL}(0, V)x^{\rho(E) - \rho(F)}.$$

Hence, for all $F \in \F$ we must show 
\begin{equation}\label{fl2}
\sum_{V \, : \, \cl(V) = F} \mu_{\mcL}(0, V) = \mu_{\F}(0,F).
\end{equation}

We once again proceed by induction on the rank of the flats of $\mcM$. 
Since $\mcM$ is loopless, $\{0\} \in \F$. Let $F = \{0\}$, then  (\ref{fl2}) follows trivially from Definition \ref{Mobius function}. Now assume (\ref{fl2}) holds true for all $F \in \F$ such that $\rho(F) \leq k-1$.
Fix a flat $F \in \F$ with $\rho(F) = k$. Then

\begin{align*}
    \mu_{\F}(0,F) &= - \sum_{F' \lneq F, F' \in \F} \mu_{\F}(0, F')\\
    &= -\sum_{F' \lneq F, F' \in \F} \sum_{V \, : \, \cl(V) = F'} \mu_{\mcL}(0,V) \\
    &= - \sum_{V \, : \,  \cl(V) \lneq F} \mu_{\mcL}(0,V) \\
    &= \sum_{V \, : \, \cl(V) = F}\mu_{\mcL}(0,V),
\end{align*}
where the second equality follows from the induction hypothesis, and the last equality follows from Definition \ref{Mobius function}. This completes the proof. 
\end{proof}

As the next theorem shows, defining the characteristic polynomial in terms of the lattice of flats of the $q$-matroid allows us to link the characteristic polynomial of a $q$-matroid with that of its projectivization matroid.

\begin{theorem}\label{char $q$-matroid/matroid}
Let $\mcM$ be a $q$-matroid and $P(\mcM)$ be its projectivization matroid. Then 
$$\chi_{\mcM}(x) = \chi_{P(\mcM)}(x).$$
\end{theorem}

\begin{proof}
 By Proposition \ref{loop/qloop} $\mcM$ contains a loop if and only if $P(\mcM)$ contains a loop. Furthermore, by Theorem \ref{flats properties}, we know $\F_{\mcM} \cong \F_{P(\mcM)}$ as lattices. Due to Propositions \ref{flats matroid} and \ref{flat $q$-matroid}, $\rho_{\mcM}(F) = h_{\F_{\mcM}}(F) = h_{\F_{P(\mcM)}}(P(F)) = r_{P(\mcM)}(P(F))$ and this for all flats $F \in \F_{\mcM}$ and $P(F) \in \F_{P(\mcM)}$. Therefore, the result follows directly from Theorems \ref{char matroid}, \ref{loop char 0} and \ref{char flats}. 
\end{proof}

We furthermore get the following result when considering the contraction of $\mcM$ by a subspace $V \leq E$.

\begin{proposition}\label{char project equiv}
Let $\mcM = (E, \rho)$ be a $q$-matroid and $P(\mcM) = (\PE, r)$ its projectivization matroid. Then for all $V \leq E$, 
$$\chi_{\mcM / V}(x) = \chi_{P(\mcM) / P(V)}(x).$$
\end{proposition}

\begin{proof}
Let $V \leq E$, then $V \in \F_{\mcM} \Leftrightarrow P(V) \in \F_{P(\mcM)}$. If $V \notin \F_{\mcM}$ then by Proposition \ref{contraction loop} $\mcM/V$ and $P(\mcM)/P(V)$ contain loops, therefore $\chi_{\mcM / V}(x) = 0 = \chi_{P(\mcM) / P(V)}(x)$.
If $V \in \F_{\mcM}$, by Theorem \ref{contraction flat iso}, $\F_{\mcM/F} \cong \F_{P(\mcM)/P(V)}$ as lattices and, by Proposition \ref{contraction loop}, both matroids are loopless. Hence, Theorems \ref{char matroid} and \ref{char flats} imply, $\chi_{\mcM /V}(x) = \chi_{P(\mcM) / P(V)}(x)$.
\end{proof}

The close connection between the characteristic polynomial of a $q$-matroid $\mcM$ and that of its projectivization matroid gives a new approach to study the former. In fact, we use this approach to find a recursive formula for the characteristic polynomial of a $q$-matroid in terms of the characteristic polynomial of its minors.
Because we will be using the recursive formula defined in Theorem \ref{char matroid} on the projectivization matroid $P(\mcM)$ and its minors, we need to pay a particular attention on whether $P(\mcM)$ and its minors contain coloops. We thus need the following few lemmas. 

\begin{lemma}\label{coloop prop}
Let $M = (S, r)$ be a matroid and $\mcM = (E, \rho)$ be a $q$-matroid. Then 

\begin{itemize}
    \item[(a)]  $w \in S$ is a coloop of $M \Longleftrightarrow  r^*(w) = 0 \Longleftrightarrow r(S - w) = r(S) - 1$
    
    \item[(b)]  $\langle w \rangle \leq E$ is a coloop of $\mcM \Longleftrightarrow \rho^*(\langle w \rangle) = 0 \Longleftrightarrow \rho(\langle w \rangle^{\perp}) = \rho(E) - 1$

\end{itemize}
\end{lemma}

\begin{proof}
Statement (a) can be found in \cite[Section 1.6, Exercise 6]{oxley}. Statement (b) follows from Definition \ref{dual $q$-matroid}.
\end{proof}

 Recall the notation $\mcQ_V = \{\langle w \rangle \, : \, \langle w \rangle \not \leq V\}$, and  $\mcQ_V^{*e} = \mcQ_V - \{\langle e \rangle\}$ for  $\langle e \rangle \in \mcQ_V$ introduced in Notation \ref{notation}.
To make the results and proofs easier to read, we may omit the brackets to denote 1-dimensional spaces and we let $v^{\perp} := \langle v \rangle^{\perp}$ for  $v \in E$.

\begin{lemma}\label{coloop}
Let $\mcM = (E, \rho)$ be a $q$-matroid, $P(\mcM) = (\PE, r)$ its projectivization matroid and $e, v \in \PE$ such that $\langle e \rangle \oplus \langle  v \rangle^{\perp} = E$. Let $A \subsetneq \mcQ_{v^{\perp}}^{*e}$, then:
\begin{itemize}
    \item[(a)] for all  $w \in \mcQ_{v^{\perp}}^{*e} - A$, the element $w$ is not a coloop of the matroid $P(\mcM)\setminus A$. 
    
    \item[(b)] for all $w \in \mcQ_{v^{\perp}}^{*e}$ and $z \in \mcQ_{v^{\perp}} - (A \cup w)$, the element $z$ is not a coloop of $P(\mcM) \setminus A /w$.
    
    \item[(c)] for all $w_1, w_2 \in \mcQ_{v^{\perp}} - A$, the matroid $P(\mcM) \setminus A/\{w_1, w_2\}$ contains a loop.
    
    \item[(d)] $e$ is a coloop of the matroid $P(\mcM) \setminus \mcQ_{v^{\perp}}^{*e}$ if and only if $\langle v \rangle$ is a coloop of $\mcM$.
\end{itemize}
\end{lemma}

\begin{proof}
Throughout, let $\mcQ := \mcQ_{v^{\perp}}^{*e}$, $H:= v^{\perp}$ and  $\{h_1, \cdots h_{n-1} \}$ be a basis of $H$. Furthermore let $A \subsetneq \mcQ$, $w \in \mcQ - A$ and consider the matroid $N := P(\mcM) \setminus A = (\PE - A, r_N) $.  Note that $r(S) = r_N(S)$ for all $S \subseteq \PE - A$.

\vspace{.1cm}
Statement (a). Since $e \notin H$ and $\dim H = \dim E  -1 $, the set $B:= \{h_1, \cdots, h_{n-1} , e\}$ is a basis of $E$, and $B \subseteq \PE - \mcQ \subseteq  \PE - (A \cup w) \subseteq \PE - A$. 
Hence by Proposition \ref{spanning set project}, $r(\PE - \mcQ) = r(\PE - (A \cup w) ) = r(\PE - A) = r(\PE)$. Because $r_N(S) = r(S)$ for all $S \subseteq \PE - A$, we have $r_N(\PE - (A \cup w) ) = r_N(\PE - A)$ so, by Lemma \ref{coloop prop}, $w$ is not a coloop of $N$, proving statement (a).

\vspace{.1cm}

Statement (b). Let $B:= \{h_1, \cdots , h_{n-1} , w\}$ which is a basis of $E$. We need only to show $r_{N/w}^*(z) \neq 0$ for all $z \in \mcQ_{v^{\perp}} - (A \cup w)$.
\begin{align*}
    r_{N/w}^*(z) &= |z| + r_{N/w}(\PE - (A\cup w \cup z)) - r_{N/w}(\PE - (A\cup w)) \\
                  &= 1 + r_{N}((\PE - (A\cup w \cup z)) \cup w) - r_{N}(w) - r_{N}((\PE -  (A \cup w)) \cup w) + r_{N}(w)\\
                  &= 1+ r(\PE - (A \cup z)) - r(\PE -  A) \\
                  &= 1 + \rho(E) - \rho(E) = 1
\end{align*}
where the last equality follows from Proposition \ref{spanning set project} because $B$ is a subset $\PE - ( A \cup z)$ and $\PE - A$. 

\vspace{.1cm}

Statement (c). Let $w_1, w_2 \in \mcQ_{v^{\perp}}-A$ and $W = \langle w_1, w_2 \rangle$. Clearly $\dim W = 2$ and  $\dim (W \cap \langle v \rangle^{\perp}) = 1$. Hence there exists $z \in W$ such that $\langle z \rangle \notin \mcQ_{v^{\perp}}$. We show $z$ is a loop of $N / \{w_1, w_2\}$, i.e. $r_{N/\{w_1, w_2\}}(z) = 0$.  
\begin{align*}
    r_{N/\{w_1, w_2\}}(z) &= r(z \cup \{w_1,w_2\}) - r(\{w_1 , w_2\}) \\
                &= \rho(\langle z, w_1,w_2 \rangle) - \rho(\langle w_1, w_2 \rangle) \\
                &= \rho(W) - \rho(W) = 0.
\end{align*}

\vspace{.1cm}

Statement (d). Let $N' := P(\mcM) \setminus \mcQ = (\PE - \mcQ, r_{N'})$. By Lemma \ref{coloop prop}  $\langle v \rangle$ is a coloop of $\mcM$ if and only if $\rho(H) = \rho(E) -1$.
Moreover, $\rho(H) = r(\PP H) = r_{N'}(\PP H)$ and $\rho(E) = r(\PE) = r(\PE - \mcQ) = r_{N'}(\PE - \mcQ)$. Hence  $\langle v \rangle$ is a coloop of $\mcM$ if and only if  $r_{N'}(\PP H) = r_{N'}(\PP H \cup e) -1$ if and only if $e$ is a coloop of $N'$.
\end{proof}

With those results in place, we are now ready to consider the first step of our main theorem. For the next results we use the following notation. Given $\mcQ_{v^{\perp}}^{*e}$, fix an ordering of its elements. Define $\mcS_0 := \emptyset$ and $\mcS_i := \{w_1, \cdots , w_i\}$ where $w_j$ is the $j^{\textup{th}}$ element of $\mcQ_{v^{\perp}}^{*e}$. Note furthermore that $|\mcS_i| = i$ and $\mcS_{q^{n-1}-1} = \mcQ_{v^{\perp}}^{*e}$. Moreover, in the proofs of the remaining results in this section, Proposition \ref{del cont commute} and Theorem \ref{char matroid} may be used without mention.

\begin{proposition}\label{intermediate form}
Let $\mcM = (E, \rho)$ be a $q$-matroid, $P(\mcM) = (\PE, r)$ its projectivization matroid, $e,v \in \PE$ such that $\langle e \rangle \oplus \langle v\rangle^{\perp} = E$.
Then 

$$\chi_{\mcM}(x) = \begin{cases} \chi_{\mcM \setminus v}(x) \cdot \chi_{\mcM / e}(x) - \sum_{i=0}^{q^{n-1}-2} \chi_{P(\mcM) \setminus S_i / w_{i+1}} & \textup{ if } v \textup{ is a coloop of } \mcM \\
 \chi_{\mcM \setminus v}(x) - \chi_{\mcM / e}(x) - \sum_{i=0}^{q^{n-1}-2} \chi_{P(\mcM) \setminus S_i / w_{i+1}} & \textup{otherwise}
\end{cases}$$
\end{proposition}

\begin{proof}
We first use an induction argument on $k := |\mcS_k|$ to show that for all $1 \leq k \leq q^{n-1} -1$,
\begin{equation}\label{int form induc}
    \chi_{P(\mcM)}(x) = \chi_{P(\mcM)\setminus \mcS_{k}}(x) - \sum_{i=0}^{k-1} \chi_{P(\mcM)\setminus \mcS_i / w_{i+1}}(x).
\end{equation}

We prove the base case when $k=1$. By Lemma \ref{coloop} (a), $w_1$ is not a coloop of $P(\mcM)$  hence, by Theorem \ref{char matroid}, $\chi_{P(\mcM)}(x) = \chi_{P(\mcM) \setminus \mcS_1}(x)  - \chi_{P(\mcM) \setminus \mcS_0 / w_1}(x)$, where recall $\mcS_0 = \emptyset$ and $\mcS_1 = \{w_1\}$.

Now assume (\ref{int form induc}) holds for $ k = q^{n-1} - 2$. Then 
\begin{align*}
    \chi_{P(\mcM)}(x) &= \chi_{P(\mcM) \setminus \mcS_k}(x) - \sum_{i=0}^{k-1} \chi_{P(\mcM)\setminus \mcS_i / w_{i+1}}(x)\\
                    &= \chi_{P(\mcM) \setminus (\mcS_k \cup w_{k+1})}(x) - \chi_{P(\mcM) \setminus \mcS_k / w_{k+1}}(x) - \sum_{i=0}^{k-1} \chi_{P(\mcM)\setminus \mcS_i / w_{i+1}}(x) \\
                    &= \chi_{P(\mcM) \setminus \mcS_{k+1}}(x) - \sum_{i=0}^{k} \chi_{P(\mcM)\setminus \mcS_i / w_{i+1}}(x).
\end{align*}
The second equality holds true by Theorem \ref{char matroid}, because $w_{k+1}$ is not a coloop of $P(\mcM)\setminus \mcS_{k}$ by \ref{coloop} (a). This establishes (\ref{int form induc}). 

Because $\mcS_{q^{n-1} -1} = \mcQ_{v^{\perp}}^{*e}$, we conclude the proof by using Theorem \ref{char matroid} on $\chi_{P(\mcM) \setminus \mcS_{q^{n-1} -1}}(x)$ with the element $e \in \mcQ_{v^{\perp}} - \mcS_{q^{n-1} -1}$.
We therefore consider two cases: when $e$ is a coloop of $P(\mcM)\setminus \mcQ_{v^{\perp}}^{*e}$ or not. By Lemma \ref{coloop} (d), those two cases correspond exactly to when $\langle v \rangle$ is a coloop of $\mcM$ or not.
First assume $e$ is a coloop of $P(\mcM)\setminus \mcQ_{v^{\perp}}^{*e}$, and therefore $\langle v \rangle$ is a coloop of $\mcM$.
Then by Theorem \ref{char matroid}
$$\chi_{P(\mcM) \setminus \mcQ_{v^{\perp}}^{*e}}(x) = \chi_{P(\mcM) \setminus \mcQ_{v^{\perp}}}(x) \chi_{P(\mcM) \setminus \mcQ_{v^{\perp}}^{*e}/e}(x).$$ 
By Theorem \ref{contraction proj}, $P(\mcM) \setminus \mcQ_{v^{\perp}} = P(\mcM \setminus v)$ and $P(\mcM) \setminus \mcQ_{v^{\perp}}^{*e}/e \cong P(\mcM / e)$ hence their respective characteristic polynomials are equal. Furthermore by Theorem \ref{char $q$-matroid/matroid}, $\chi_{P(\mcM \setminus v)}(x) = \chi_{\mcM \setminus v}(x)$ and $\chi_{P(\mcM/e)}(x) = \chi_{\mcM/e}(x)$. Therefore,
$\chi_{P(\mcM) \setminus \mcS_{q^{n-1}-1}}(x) = \chi_{P(\mcM) \setminus \mcQ_{v^{\perp}}}(x) \cdot \chi_{P(\mcM) \setminus \mcQ_{v^{\perp}}^{*e}/e}(x)  = \chi_{\mcM \setminus v}(x) \cdot \chi_{\mcM / e}(x)$, which, when substituted in (\ref{int form induc}) gives us the wanted equality.

If $e$ is a not coloop of $P(\mcM) \setminus \mcQ_{v^{\perp}}^{*e}$, and therefore $\langle v \rangle $ is not a coloop of $\mcM$. Then $$\chi_{P(\mcM) \setminus \mcQ_{v^{\perp}}^{*e}}(x) = \chi_{P(\mcM) \setminus \mcQ_{v^{\perp}}}(x) - \chi_{P(\mcM) \setminus \mcQ_{v^{\perp}}^{*e}/e}(x).$$
Once again using Theorems \ref{contraction proj} and \ref{char $q$-matroid/matroid}, the wanted equality follows. 
\end{proof}

At this point, note that the characteristic polynomial $\chi_{\mcM}(x)$ depends on both the characteristic polynomial of minors of the $q$-matroid $\mcM$ and the characteristic polynomial of minors of $P(\mcM)$. In the following Theorem, we rewrite all characteristic polynomials of minors of $P(\mcM)$ in terms of characteristic polynomials of minors of $\mcM$.

\begin{theorem}
Let $\mcM = (E, \rho)$ be a $q$-matroid and $e,v \in E$ such that $\langle e \rangle \oplus \langle v\rangle ^{\perp} = E$.
Then 

$$\chi_{\mcM}(x) = \begin{cases}
\chi_{\mcM \setminus v}(x) \cdot \chi_{\mcM/e}(x) - \sum_{w \in \mcQ_{v^{\perp}}^{*e}} \chi_{\mcM/w} (x)& \textup{ if } v \textup{ is a coloop of } \mcM, \\
\chi_{\mcM \setminus v}(x) - \sum_{w \in \mcQ_{v^{\perp}}} \chi_{\mcM/w} (x) & \textup{otherwise.}
\end{cases}$$
\end{theorem}

\begin{proof}
Given the equation of Proposition \ref{intermediate form}, we show that

\begin{equation}\label{final form}
\sum_{i=0}^{q^{n-1}-2} \chi_{P(\mcM) \setminus \mcS_i / w_{i+1}}(x) = \sum_{w \in \mcQ_{v^{\perp}}^{*e}} \chi_{\mcM/ w } (x).
\end{equation}

Fix  $0 \leq i \leq q^{n-1}-2$ and consider $\chi_{P(\mcM)\setminus S_i /w_{i+1}}(x)$.  For all $A \subseteq \mcQ_{v^{\perp}} - (S_i \cup w_{i+1})$, we show by induction on $|A|$ that 
\begin{equation}\label{induc final form}
\chi_{P(\mcM)\setminus S_i /w_{i+1}}(x) = \chi_{P(\mcM)\setminus (S_i \cup A) /w_{i+1}}(x).
\end{equation} 
First let $|A| = 1$ and let $w \in A$.
By Lemma \ref{coloop} (b), $w$ is not a coloop of $P(\mcM)\setminus \mcS_i/w_{i+1}$. Therefore
$$\chi_{P(\mcM)\setminus S_i /w_{i+1}}(x) =  \chi_{P(\mcM)\setminus (S_i \cup w) /w_{i+1}}(x) - \chi_{P(\mcM)\setminus S_i /\{w_{i+1} , w\}}(x).$$
By Lemma \ref{coloop} (c), $P(\mcM) \setminus S_i \, / \{w_{i+1},w\}$ contains a loop which implies its characteristic polynomial is $0$. Hence $\chi_{P(\mcM)\setminus S_i /w_{i+1}}(x) =  \chi_{P(\mcM)\setminus (S_i \cup w) /w_{i+1}}(x)$. 
Now let $|A| = k$ and let $w \in A$. Since $|A - w| = k-1$ by induction hypothesis we get $\chi_{P(\mcM)\setminus S_i /w_{i+1}}(x) = \chi_{P(\mcM)\setminus (S_i \cup (A-w)) /w_{i+1}}(x)$. Once again by Lemma \ref{coloop} (b), $w$ is not a coloop of $P(\mcM)\setminus (\mcS_i \cup (A-w))/w_{i+1}$ hence $$ \chi_{P(\mcM)\setminus (S_i \cup (A-w)) /w_{i+1}}(x) =  \chi_{P(\mcM)\setminus (S_i \cup A) /w_{i+1}}(x) -  \chi_{P(\mcM)\setminus (S_i \cup A) /\{w_{i+1}, w\}}(x).$$ 
By Lemma \ref{coloop} (c) $P(\mcM)\setminus (S_i \cup A) /\{w_{i+1}, w\}$ contains a loop and therefore its characteristic polynomial is 0. This completes the proof of (\ref{induc final form}), which if $A = \mcQ_{v^{\perp}} - (S_i \cup w_{i+1})$ shows that
$$\chi_{P(\mcM)\setminus S_i /w_{i+1}}(x) =\chi_{P(\mcM) \setminus \mcQ_{v^{\perp}}^{*w_{i+1}} /w_{i+1}}(x).$$

Finally  by Theorems \ref{contraction proj} (where $S = \{w_{i+1}\}$) and \ref{char $q$-matroid/matroid} we get
\begin{equation}\label{equ rewrting sum}
    \chi_{P(\mcM)\setminus S_i /w_{i+1}}(x) = \chi_{\mcM/w_{i+1}}(x).
\end{equation}
Since the above induction holds true for any $i$ chosen, then (\ref{equ rewrting sum}) holds for all $0 \leq i \leq q^{n-1}-2$ and
\begin{align*}
    \sum_{i=0}^{q^{n-1}-2} \chi_{P(\mcM) \setminus \mcS_i / w_{i+1}}(x) &= \sum_{i=0}^{q^{n-1}-2} \chi_{\mcM/w_{i+1}}(x) \\
    &= \sum_{w \in \mcQ_{v^{\perp}}^{*e}} \chi_{\mcM/w} (x).
\end{align*}

Substituting (\ref{final form}) into the equation of Proposition \ref{intermediate form} gives the desired result.
\end{proof}

\section{Rank Metric and Linear Block Codes}

The study of matroids and $q$-matroids plays an important role in coding theory. In fact, it is well known that $\FF_{q^m}$-linear rank metric codes induce $q$-matroids, and that linear block codes with the Hamming metric give rise to matroids. Furthermore many of the code invariants can be determined from the induced ($q$-)matroid. 
In \cite{ABNR21}, Alfarano and co-authors showed that an $\FF_{q^m}$-linear rank metric code $\C$ induces a linear block code that shares similar parameters. We show in this section how the projectivization matroid of a $q$-matroid relates to the matroid associated to that linear block code. Furthermore we use this relation and results from Section 5 to show the $q$-analogue of the critical theorem in terms of $q$-matroids and $\FF_{q^m}$-linear rank metric codes. 

We start the section by recalling some coding theory concepts. Throughout, let $\FF_{q^m}$ be a field extension of $\FF_q$ of degree $m$. Furthermore let $\Gamma$ be a basis of the vector space $\FF_{q^m}$ over $\FF_q$. For all $v = (v_1, \cdots v_n) \in \FF_{q^m}^n$, let $\Gamma(v)$ be the $n \times m$ matrix such that the $i^{\textup{th}}$ row of $\Gamma(v)$ is the coordinate  vector of $v_i$ with respect to the basis $\Gamma$.  Let $\rk_{\FF}(-)$, $\colsp_{\FF}(-), \rowsp_{\FF}(-)$ respectively denote the rank, column space and row space of a matrix over the field $\FF$. Moreover, throughout the section, we fix the NSBF on $\FF_q^n$ to be the standard dot product. We define the following two weight functions. 

\begin{definition}
For all $v \in \FF_{q^m}$, the \emph{Hamming weight} $\omega_H$, and the \emph{rank weight} $\omega_R$ of $v$ are defined as follow:
\begin{align*}
    \omega_H(v) &:= \# \textup{non-zero entries of }v  \\
    \omega_R(v) &:= \rk_{\FF_q} (\Gamma(v)).
\end{align*}
\end{definition}
It is well known that the rank weight is independent of the basis $\Gamma$ chosen.
Both weight functions induce a metric on $\FF_{q^m}^n$, where  $d_{\Delta}(v, w) = \omega_{\Delta}(v - w)$ for $v, w \in \FF_{q^m}$ and $\Delta \in \{H, R\}$.
A \emph{linear block code} is a subspace of the metric space $(\FF_{q^m}^n, d_H)$, and an \emph{$\FF_{q^m}$-linear rank metric code} is a subspace of the metric space $(\FF_{q^m}^n, d_R)$.
Given a code  $\C \leq \FF_{q^m}^n$, let $d_H(\C)$, respectively $d_R(\C)$, denote the minimum weight over all non zero elements of the code.
If $\dim \C = k$, then $d_{\Delta}(\C) \leq n - k +1$ for $\Delta \in \{H, R\}.$
For each metric, the above bound is called the \emph{Singleton bound}. $\C$ is said to be \emph{maximum distance separable}, respectively \emph{maximum rank distance} if the Hamming-metric, respectively the rank-metric, Singleton bound is achieved.
For both metrics, the weight distribution of the code is defined as follows.

\begin{definition}
Let $\C \leq \FF_{q^m}^n$ be a code. For $\Delta \in \{H, R\}$, let 
$$W_{\Delta}^{(i)}(\C) := \left| \{v \in \C \, : \omega_{\Delta}(v) = i\}\right|.$$
Furthermore let $W_{\Delta}(\C) := (W_{\Delta}^{(i)}(\C) \, : 0 \leq i \leq n)$.  $W_{H}(\C)$ is called the \emph{Hamming weight distribution} of $\C$ and $W_R(\C)$ is called the \emph{rank weight distribution} of $\C$.  
\end{definition}

We now introduce two notions of support for elements of $\FF_{q^m}^n$.

\begin{definition}\label{support}
Let $v = (v_1 , \cdots , v_n) \in \FF_{q^m}^n$ and $V \subseteq \FF_{q^m}^n$. 
\begin{align*}
    S_{H}(v) &= \{ i \, : v_i \neq 0\} \quad \textup{and} \quad S_H(V) = \bigcup_{v \in V} S_H(v) \\
    S_R(v) &= \colsp_{\FF_q}(\Gamma(v)) \quad \textup{and} \quad S_R(V) = \sum_{v \in V} S_R(v)
\end{align*}

$S_H$, respectively $S_R$, are called the \emph{Hamming support} and \emph{rank support} of $v \in \FF_{q^m}^n$ or $V \subseteq \FF_{q^m}^n$. 
\end{definition}

Once again, the rank support of an element $v \in \FF_{q^m}^n$, or subset $V \subseteq \FF_{q^m}$, is independent of the basis $\Gamma$ chosen. Furthermore, a linear block code, respectively $\FF_{q^m}$-linear rank metric code, $\C \leq \FF_{q^m}^n$ is said to be \emph{non-degenerate} if $S_H(\C) = [n]$, respectively $S_R(\C) = \FF_q^n$.
Given a code $\C \leq \FF_{q^m}^n$, it is of interest to consider the set of elements of $\C$ with a given support.

\begin{definition}
Let $\C \leq \FF_{q^m}^n$ be a code,  $A \subseteq [n]$ and $V \leq \FF_q^n$. Let 
\begin{align*}
\C_H(A) &:= \{v \in \C \, : \, S_H(v) = A\} \\
\C_R(V) &:= \{v \in \C \, : \, S_R(v) = V\}
\end{align*}
\end{definition}

We now make the connection between codes and ($q$-)matroids. Note that linear block codes and $\FF_{q^m}$-linear rank metric codes can be represented via a generating metric $G \in \FF_{q^m}^{k \times n}$, where $\C := \rowsp_{\FF_{q^m}}(G)$.
Both a matroid and a $q$-matroid can be induced from the generating matrix. The following construction is well known for matroids (see \cite[Sec.6]{oxley}) and has been established in \cite[Sec. 5]{JP18} for $q$-matroids. 

\begin{proposition}\label{rank representable}
Let $\C \leq \FF_{q^m}^n$ be a code and $G \in \FF_{q^m}^{k \times n}$ a generating matrix of $\C$. For $i \in [n]$, let $e_i \in \FF_q^n$ denote the $i^{th}$ standard basis vector and for $V \leq \FF_q^n$ let $Y_V \in \FF_q^{n \times t}$ such that $\colsp_{\FF_q}(Y_V) = V$. Define $r: [n] \rightarrow \NN_0$ and $\rho: \mcL(\FF_q^n) \rightarrow \NN_0$ such that:
\begin{align*}
     r(A) &= \rk_{\FF_{q^m}} \left( G \cdot \begin{bmatrix} e_{i_1} &\cdots & e_{i_a} \end{bmatrix} \right) \textup{ for all } A \subseteq [n]\\
    \rho(V) &= \rk_{\FF_{q^m}}\left(G \cdot Y_V \right) \textup{ for all } V \leq \FF_q^n.
\end{align*}

Then $M_{\C} := ([n], r)$ is a matroid and $\mcM_{\C} = (\FF_q^n, \rho)$ is a $q$-matroid and are called the matroid (resp.\ $q$-matroid) associated with $\C$. 
\end{proposition}
Note that neither $M_{\C}$ nor $\mcM_{\C}$ depend on the choice of generating matrix for $\C$.
Furthermore a matroid or $q$-matroid is $\FF_{q^m}$-representable if it is induced by a linear block code or rank metric code $\C \leq \FF_{q^m}^n$.
The ($q$-)matroid induced by a code is a useful tool to determine some of the code's invariants. In the rest of this section, we consider invariants of the code that are determined by the characteristic polynomial of the induce ($q$-)matroid.
We  first recall the notion of weight enumerator of a ($q$-)matroid. The weight enumerator of the $q$-matroid was defined in \cite[Def. 43]{BCIJ21} and a similar concept was established in \cite{Greene} for matroids.

\begin{definition}\label{weight distribution matroid}
Let $M = (S, r)$ be a matroid and $\mcM = (E, \rho)$ be a $q$-matroid, with $|S| = n = \dim E$.  Let
\begin{align*}
    A_M^{(i)}(x) &= \sum_{A \subseteq S, |A| = i} \chi_{M/(S-A)}(x)\\
    A_{\mcM}^{(i)}(x) &= \sum_{V \leq E, \dim V = i} \chi_{\mcM/V^{\perp}}(x)
\end{align*}
The \emph{weight enumerator} of the matroid, respectively $q$-matroid, is the list $A_M := (A_M^{(i)}(x) \, : \, 1 \leq i \leq n)$, respectively $A_{\mcM} := (A_{\mcM}^{(i)}(x) \, : \, 1 \leq i \leq n)$.  
\end{definition}

Note in the above equations that if $S-A$ or $V^{\perp}$ are not flats of their respective matroid or $q$-matroid, then $\chi_{M/(S-A)}(x) = 0 = \chi_{\mcM/V^{\perp}}(x)$. Hence the summands of the weight enumerator can be restricted to the complement, respectively the orthogonal space, of flats. Thus we have the following result, which is well-known for matroids (see \cite[Prop 3.3]{Greene}) and was hinted at in \cite{BCIJ21} for $q$-matroids.

\begin{theorem}\label{weight enum flats}
Let $M = (S, r)$ be a matroid and $\mcM = (E, \rho)$ be a $q$-matroid, with $|S| = n = \dim E$. Then
\begin{align*}
    A_M^{(i)}(x) &= \sum_{F \in \F_M, |F| =n- i} \chi_{M/F}(x),\\
    A_{\mcM}^{(i)}(x) &= \sum_{F \in \F_{\mcM}, \dim F = n - i} \chi_{\mcM/F}(x).
    \end{align*}
\end{theorem}

It was shown in \cite[Lem 49]{BCIJ21} that the weight enumerator of a representable $q$-matroid is closely related to the weight distribution of its associated code. For matroids a similar relation holds and was established in \cite[Prop 3.2]{Greene}.

\begin{theorem}\label{weight distribution equality}
Let $\C \leq \FF_{q^m}^n$ be a code, $M$ and $\mcM$ be, respectively, the matroid and $q$-matroid induced by $\C$. Let $A \subseteq [n]$ and $V \leq \FF_q^n$. Then 
\begin{itemize}
    \item[(1)] $\chi_{M/A}(q^m) = |\C_H([n]-A)|$ \quad and \quad $W_H^{(i)}(\C) = A_{M}^{(i)}(q^m)$.
    \item[(2)]$ \chi_{\mcM/V}(q^m) = |\C_R(V^{\perp})|$ \quad and \quad $W_R^{(i)}(\C) = A_{\mcM}^{(i)}(q^m).$
\end{itemize}

\end{theorem}

\begin{remark}\label{support/flats}
The above Theorem together with Proposition \ref{weight enum flats}, tells us that if $|\C_H([n] -A)|$ and $|\C_R(V^{\perp})|$ are non-zero then $A$ and $V$ are flats in $M$ and $\mcM$, respectively. 
\end{remark}

Because the weight enumerator of a ($q$-)matroid can be expressed in terms of flats, we get the following relation between the weight enumerator of a $q$-matroid and that of its projectivization matroid. 

\begin{proposition}\label{weight enum q-matroid/proj}
Let $\mcM = (E, \rho)$ be a q-matroid and $P(\mcM) = (\PE , r)$ its projectivization matroid. Then 
\begin{equation}
    A_{P(\mcM)}^{(j)}(x) = \begin{cases}  A_{\mcM}^{(i)}(x) & \textup{if } j = \frac{q^n - q^{n-i}}{q-1}, \\
    0 & \textup{otherwise}.
    \end{cases}
\end{equation}
\end{proposition}

\begin{proof}
First, recall by Lemma \ref{flats properties}, $\F_{P(\mcM)} = \{P(F) \,: \, F \in \F_{\mcM}\}$ and $\dim F =n- i \Leftrightarrow |P(F)| = \frac{q^{n-i} -1}{q-1} \Leftrightarrow |\PE - P(F)| = \frac{q^n - q^{n-i}}{q-1} $. Furthermore, by Proposition \ref{char project equiv}, if $X \subseteq \PE$ is not a flat, then $\chi_{P(\mcM)/X}(x) = 0$. So for all $1 \leq j \leq \frac{q^n-1}{q-1}$, such that $j \neq \frac{q^n - q^{n-i}}{q-1}$ for some $1 \leq i \leq n$,  we get $A_{P(\mcM)}^{(j)}(x) = 0$. 
Now assume $j = \frac{q^n-q^{n-i}}{q-1}$ for some $1 \leq i \leq n$. Then 
\begin{align*}
    A_{\mcM}^{(i)}(x) &= \sum_{F \in \F_{\mcM}, \dim F = n-i} \chi_{\mcM/F}(x) \\
                    &= \sum_{F \in \F_{\mcM}, \dim F = n-i} \chi_{P(\mcM)/P(F)}(x)\\
                    &= \sum_{P(F) \in \F_{P(\mcM)}, |P(F)| = \frac{q^{n-i}-1}{q-1}} \chi_{P(\mcM)/P(F)}(x)\\
                    &= A_{P(\mcM)}^{(j)}(x),
\end{align*}
where the second equality follows from Proposition \ref{char project equiv}.
\end{proof}

With the above setup, we now discuss the linear block code induced by an $\FF_{q^m}$-linear rank metric code, as introduced in \cite{ABNR21}. In their paper, the authors use $q$-systems and projective systems to introduce the Hamming-metric code associated to a rank metric code.
We use a slightly different approach to introduce the associated Hamming metric code that does not require the previously stated notions. 

\begin{definition}\label{Fq decomp}
Let $\C \leq \FF_{q^m}^n$ be an $\FF_{q^m}$-linear rank metric code and let $G \in \FF_{q^m}^{k\times n}$ be a generating matrix of $\C$. Furthermore let $H \in \FF_q^{n \times \frac{q^n-1}{q-1}}$ where each column of $H$ is a representative of a distinct element of $\PF_q^n$. We call the matrix $G^H := G \cdot H$  an \emph{$\FF_q$-decomposition} of $G$ via $H$ and $\C^H:= \rowsp_{\FF_{q^m}}(G^H)$ is called a \emph{Hamming-metric code associated to $\C$ via $H$}
\end{definition}

\begin{remark}

Given a non-degenerate $\FF_{q^m}$-rank metric code $\C$, the code $\C^H$ of Definition \ref{Fq decomp} is a Hamming-metric code associated with $\C$ as in \cite[Def. 4.5]{ABNR21}. In fact, it easy to show  the projective system induced by the columns of $G^H$ is a representative of the equivalence class $(\mathrm{Ext}^H \circ \Phi)([\C])$ as introduced in \cite{ABNR21}.
Furthermore note that unlike the construction in \cite{ABNR21}, Definition \ref{Fq decomp} does not depend on $q$-systems and projective systems hence we do not require $\C$ to be non-degenerate code.
Finally, as noted in \cite{ABNR21}, a Hamming metric code associated to an $\FF_{q^m}$-linear rank metric code $\C$ is not unique. In our case, $\C^H$ depends on the choice of the matrix $H$ of the $\FF_q$-decomposition. However, all Hamming-metric codes associated with $\C$ are monomially equivalent.

\end{remark}

Because $\C^H$ is a linear block code, it induces a matroid $M_{\C^H}$. It turns out that $M_{\C^H}$ is equivalent to the projectivization matroid $P(\mcM_{\C})$ of the $q$-matroid induced by $\C$.

\begin{theorem}\label{representable project}
Let $\C \leq \FF_{q^m}^n$ be a rank metric code, $\mcM_{\C}$  its associated $q$-matroid, and $P(\mcM_{\C})$ its projectivization matroid. Furthermore let $\C^H$ be a Hamming-metric code associated to $\C$ via $H$ and $M_{\C^H}$ its induced matroid. Then 
$$P(\mcM_{\C}) \cong M_{\C^H} \quad \textit{as  matroids}.$$
\end{theorem}

\begin{proof}
Let $G \in \FF_{q^m}^{k \times n}$ be a generator matrix of $\C$ and let $G^H = G \cdot H$ be a $\FF_q$-decomposition of $G$ via $H$. 
Let $h_i$ be the $i^{th}$ column of $H$, hence $\PF_q^n = \{ \langle h_i \rangle \, : \, i \in \left[ \frac{q^n-1}{q-1} \right]$. Define the bijection $\psi : \PF_q^n \rightarrow \left[\frac{q^n-1}{q-1} \right]$, where $\psi(\langle h_i \rangle) = i$. 
Furthermore, let $A \subseteq \PF_q^n$, $\psi(A) = \{i_1, \cdots i_a\}$ and $Y_A := \begin{bmatrix} e_{i_1} & \cdots & e_{i_a} \end{bmatrix} \in \FF_q^{\frac{q^n-1}{q-1} \times |A|}$, where $e_j$ is the $j^{\textup{th}}$ standard basis element of $\FF_q^{\frac{q^n-1}{q-1}}$.
By Proposition \ref{rank representable} we get

\begin{align*}
    r_{M_{\C^H}}(\psi(A)) &= \rk_{\FF_{q^m}}(G^H \cdot \begin{bmatrix} e_{i_1}& \cdots & e_{i_a} \end{bmatrix})\\
    &= \rk_{\FF_{q^m}}(G \cdot H \cdot \begin{bmatrix} e_{i_1}& \cdots & e_{i_{a}} \end{bmatrix})\\
    &=\rk_{\FF_{q^m}}(G \cdot \begin{bmatrix} h_{i_1} & \cdots & h_{i_a} \end{bmatrix}) \\
    &= \rho_{\mcM_{\C}}( \langle h_{i_1}, \cdots , h_{i_a} \rangle_{\FF_q})\\
    &=r_{P(\mcM_{\C})}(A),
\end{align*}
where the last equality follows from Theorem \ref{$q$-matroid/matroid}.
\end{proof}

\begin{remark}\label{relabel}
The above theorem allows us to relabel the groundset $\left[\frac{q^n-1}{q-1} \right]$ of the matroid $M_{\C^H}$ in terms of the elements of the projective space $\PF_q^n$. Precisely if $G \cdot H$ is the $\FF_q$-decomposition associated to $\C^H$, relabel $i \in \left[\frac{q^n-1}{q-1} \right]$ by $\langle h_i \rangle \in \PF_q^n$, where $h_i$ is the $i^{\textup{th}}$ column of $H$.
\end{remark}

We get the following result as an immediate corollary of Theorem \ref{representable project}. 

\begin{corollary}
If $\mcM$ is $\FF_{q^m}$-representable then its projectivization matroid $P(\mcM)$ is $\FF_{q^m}$-representable. 
\end{corollary}

In \cite[Theorem 4.8]{ABNR21}, it was established that the rank-weight distribution of a rank metric code $\C$ is closely related to the Hamming-weight distribution of any Hamming-metric code associated to $\C$. By Theorem \ref{weight distribution equality}, Proposition \ref{weight enum q-matroid/proj} and Theorem \ref{representable project} we arrive at the same result from a purely matroid/$q$-matroid approach. 

\begin{theorem}\cite[Thm 4.8]{ABNR21}
Let $\C$ be an $\FF_{q^m}$-linear rank metric code and $\C^H$ be a Hamming-metric code associated to $\C$.Then 

$$W_H^{(j)}(\C^H) = \begin{cases} W_R^{(i)}(\C) & \textup{if } j = \frac{q^n - q^{n-i}}{q-1},\\
                                        0 & \textup{otherwise.}
\end{cases}$$
\end{theorem}

We conclude the paper by showing the $q$-analogue of the critical Theorem. The critical Theorem, introduced by Crapo and Rota \cite[Thm 1]{crapo1970foundations}, states that the characteristic polynomial of the matroid $M_{\C}$ induced by the linear block code $\C$ determines the number of multisets of codewords with a given support. It was nicely restated in \cite[Thm 2]{BRITZ200555} in terms of coding theory terminology (recall Definition \ref{support}).

\begin{theorem}\label{critical theorem}
Let $\C \leq \FF_{q}^n$ be a linear block code and $M = ([n], r)$ its induced matroid. For all $A \subseteq [n]$, the number of ordered $t$-tuples $V = (v_1, \cdots , v_t)$, where $v_j \in \C$ for all $1 \leq j \leq t$, such that $S_H(V) = A$ is given by $\chi_{M/([n]-A)}(q^t)$.  
\end{theorem}

For our last result, we show an analogous statement for $\FF_{q^m}$-linear rank metric codes and $q$-matroids by using the projectivization matroid. It is worth mentioning that Alfarano and Byrne were able to show an analogue of the critical theorem for $q$-polymatroids and matrix rank metric codes by using a different approach involving the M\"obius inversion formula  \cite{ABcrit} . 

For the next results we make use of Remark \ref{relabel}, and relabel the elements of the groundset of the matroid induced by $\C^H$ in terms of elements of $\PF_q^n$. Following this relabeling, we can also describe the support of a codeword of $\C^H$ in terms of  the elements of $\PF_q^n$. More precisely, if $\C^H$ is induced by the $\FF_q$-decomposition  $G \cdot H$, for any $v \in \C^H$, let $S_H(v) = \{ \langle h_i \rangle \in \PF_q^n \, : \, v_i \neq 0\}$, where $ h_i$ and $v_i$ are respectively the $i^{\textup{th}}$ column of $H$ and the $i^{\textup{th}}$ component of $v$. Furthermore, we need the following well-known result for which we include a proof for self-containment. For two vectors $v, w$ we let $v \cdot w$ denote the standard dot-product.

\begin{lemma}\label{dual support}
Let $v \in \FF_{q^m}^n$,  $S_R(v) = W \leq \FF_q^n$ and $w \in \FF_q^n$. Then $v \cdot w = 0$ if and only if $w \in W^{\perp}$.
\end{lemma}

\begin{proof}
Let $\Gamma := \{\gamma_1 \, \cdots, \gamma_m\}$ be a basis of $\FF_{q^m}$ over $\FF_q$, and let $Y:= \Gamma(v) \in \FF_q^{n \times m}$, where $W:= \colsp_{\FF_q}(Y)$. Then $v \cdot w = 0 \Leftrightarrow \sum_{i=1}^n v_iw_i = 0 \Leftrightarrow \sum_{i=1}^n \left( \sum_{j=1}^m \gamma_{j}v_{ij} \right)w_i = 0 \Leftrightarrow \sum_{j=1}^m \gamma_j \left(\sum_{i=1}^n v_{ij}w_i\right) = 0$. Since $\Gamma$ is a basis of $\FF_{q^m}$ over $\FF_q$ and $v_{ij}w_i \in \FF_q$, the previous equality holds if and only if $\sum_{i=1}^n v_{ij}w_i = 0$ for all $1 \leq j \leq m$. But note $\sum_{i=1}^n v_{ij}w_i = v^{(j)} \cdot w$, where $v^{(j)}$ is the $j^{\textup{th}}$ column of $Y$. Hence $v \cdot w = 0 \Leftrightarrow v^{(j)} \cdot w = 0$ for all $1 \leq j \leq m \Leftrightarrow w \in \colsp_{\FF_q}(Y)^{\perp} \Leftrightarrow w \in W^{\perp}$.
\end{proof}

The following Lemma relates the rank support of elements of the code $\C$ with the Hamming support of elements of the associated Hamming-metric code $\C^H$.

\begin{lemma}\label{rank sup/ ham sup}
Let $\C \leq \FF_{q^m}^n$ be a rank metric code, $\C^H$ be a Hamming-metric code associated to $\C$ via $H$. Furthermore, let $V = \{v_1, \cdots , v_t\}$ be a subset of $\C$ and $V \cdot H:= \{v_1 \cdot H, \cdots, v_t \cdot H\}$ . Then 
$$S_R(V) = W \Leftrightarrow S_H(V \cdot H) = \PF_q^n - P( W^{\perp}).$$
Moreover if $\mcM$ and $P(\mcM)$ are the q-matroid and projectivization matroid induced respectively by $\C$ and $\C^H$ then $W^{\perp}$ and $P( W^{\perp})$ are, respectively, flats of $\mcM$ and $P(\mcM)$.
\end{lemma}

\begin{proof}
Consider the subset $V := \{v_1, \cdots , v_t\} \subseteq \C$. By definition, $S_R(V) = \sum_{j=1}^t S_R(v_j) =:  W$.
Let $W_j := S_R(v_j)$. For all $1 \leq j \leq t$,  by Lemma \ref{dual support}, $v_j \cdot w = 0$ if and only if $w \in W_j^{\perp}$ . Hence for all columns $h_i$ of $H$, it follows that $v_j \cdot h_i = 0 $ if and only if $ h_i \in W_j^{\perp}$. By definition, this is true if and only if $S_H(v_j \cdot H) = \PF_q^n - P(W_j^{\perp})$.
Hence $S_H( V \cdot H) = \bigcup_{i=1}^t S_H(v_j \cdot H) = \bigcup_{j=1}^t (\PF_q^n - P(W_j^{\perp})),$ where the first equality follows by definition.
Therefore $S_R(V) = W \Leftrightarrow S_H(V \cdot H) = \bigcup_{j=1}^t (\PF_q^n - P(W_j^{\perp})) = \PF_q^n - (\bigcap_{j=1}^t P(W_j^{\perp})) = \PF_q^n - P(W^{\perp})$.
Finally, $W^{\perp}$ and $P(W^{\perp})$ are flats of $\mcM$ and $P(\mcM)$ respectively because of Remark \ref{support/flats}.
\end{proof}

We are now ready to show the critical theorem for $\FF_{q^m}$-linear rank metric codes and $q$-matroids.

\begin{theorem}
Let $\C \leq \FF_{q^m}^n$ be an $\FF_{q^m}$-linear rank metric code and $\mcM$ its induced $q$-matroid. For all $W \leq \FF_q^n$, the number of ordered $t$-tuples $V = (v_1, \cdots , v_t)$, where $v_j \in \C$ for all $1 \leq j \leq t$, such that $S_R(V) = W$ is given by $\chi_{\mcM/W^{\perp}}(q^{mt})$. 
\end{theorem}

\begin{proof}
Let $\C^H$ be the Hamming-metric code associated with $\C$ via $H$ and $P(\mcM)$ be its associated matroid. Note that every element of $\C^H$ is of the form $v \cdot H$ for some $v \in \C$. Hence every tuple of elements of $\C^H$ is of the form $V \cdot H$ for some $V \subseteq \C$. Furthermore, since $H$ has full-row rank, there is a bijection between elements of $\C$ and $\C^H$ and hence a bijection between $t$-tuples $V \subseteq \C$ and $t$-tuples $V \cdot H \subseteq \C^H$. By Theorem \ref{critical theorem}, the number of $t$-tuple $V \cdot H \subseteq \C^H$ such that $S_H(V \cdot H) = \PE - P(W^{\perp})$ is given by $\chi_{P(\mcM) / P(W^{\perp})}(q^{mt})$.  Moreover, by Proposition \ref{char project equiv}, $\chi_{P(\mcM)/P(W^{\perp})}(q^{mt}) = \chi_{\mcM / W^{\perp}}(q^{mt})$. 
Finally, by Lemma \ref{rank sup/ ham sup}, $S_R(V \cdot H) = \PE - P(W^{\perp})$ if and only if $S_R(V) = W$ and therefore $\chi_{\mcM / W^{\perp}}(q^{mt})$ counts the number of $t$-tuples $V \subseteq \C$ such that $S_R(V) = W$. 
\end{proof}

The above result, and its proof, shows a close connection between the critical theorem for matroids and that for $q$-matroids. Furthermore it can easily been seen from the above approach that the critical problem for $q$-matroids, that is finding the smallest power of $q$ that makes the characteristic polynomial of a $q$-matroid non-zero, is a specific case of the critical problem for matroids. 

\section*{Further questions}
Here are a few questions that arise from our work:
\begin{itemize}
    \item If a projectivization matroid $P(\mcM)$ is $\FF_{q^m}$-representable then is the $q$-matroid $\mcM$ representable?
    
    \item It is well known that the characteristic polynomial of a matroid can be derived from the Tutte polynomial of that matroid. Hence the characteristic polynomial of a $q$-matroid $\mcM$ can also be derived from the Tutte polynomial of the projectivization matroid $P(\mcM)$. Can invariants of $\FF_{q^m}$-rank metric codes be determined from the Tutte polynomial of the projectivization matroid associated to the code.
\end{itemize}

\bibliographystyle{plain}
\bibliography{projmatroid}

\end{document}